\documentclass[a4paper, 11pt, reqno]{amsart}
\usepackage[utf8]{inputenc}
\usepackage{amssymb, amsmath, amsfonts, amsthm,  enumerate, booktabs, mathrsfs,tikz}
\usepackage[hidelinks]{hyperref}
\makeatother
\newcommand{\ol}{\overline}
\newcommand{\ul}{\underline}

\newtheorem{theorem}{Theorem}[section]
\newtheorem{corollary}[theorem]{Corollary}
\newtheorem{lemma}[theorem]{Lemma}
\newtheorem{proposition}[theorem]{Proposition}
\newtheorem{question}[theorem]{Question}

\newtheorem{remark}[theorem]{Remark}
\theoremstyle{definition}
\newtheorem{definition}[theorem]{Definition}
\newtheorem{example}[theorem]{Example}

\newcommand{\N}{\mathbb{N}}
\newcommand{\R}{\mathbb{R}}
\newcommand{\F}{\mathbb{F}}
\newcommand{\T}{\mathbb{T}}

\newcommand{\cB}{\mathcal{B}}
\newcommand{\cC}{\mathcal{C}}
\newcommand{\cG}{\mathcal{G}}
\newcommand{\cR}{\mathrel{\mathcal{R}}}
\newcommand{\cL}{\mathrel{\mathcal{L}}}
\newcommand{\cH}{\mathrel{\mathcal{H}}}
\newcommand{\cD}{\mathrel{\mathcal{D}}}
\newcommand{\cJ}{\mathrel{\mathcal{J}}}
\newcommand{\cS}{\mathcal{S}}

\newcommand{\cZ}{\mathcal{Z}}

\title[Sch\"utzenberger groups of tropical matrices]{The Sch\"utzenberger groups and maximal subgroups of tropical matrices}
\begin{document}
\keywords{tropical matrices, maximal subgroups, Sch\"utzenberger groups}
\maketitle
\begin{center}
THOMAS AIRD\footnote{Email \texttt{Thomas.Aird@manchester.ac.uk}.}
\\ \ \\
Department of Mathematics, University of Manchester, \\
Manchester M13 9PL, UK. \\
Heilbronn Institute for Mathematical Research, Bristol, UK.
\end{center}
\begin{abstract}
    We classify the Sch\"utzenberger groups of the category of matrices over the tropical semiring, $M(\T)$, in doing so, we obtain a classification for the Sch\"utzenberger groups of the semigroupoid of matrices over the finitary tropical semiring, $M(\F\T)$.
    We then classify the maximal subgroups of the monoid of $n \times n$ matrices over the tropical semiring, $M_n(\T)$, for all $n \in \N$; generalising a result in the literature and correcting an erroneous proof. We proceed to show that for some $n \in \N$ there exists a group which appears as a Sch\"utzenberger group of $M_n(\T)$ but does not appear as a maximal subgroup.
\end{abstract}
\section{Introduction} 
The tropical semiring, $\T$, is the set of real numbers with a $-\infty$ element adjoined under the operations of maximum and addition. Research into tropical mathematics has received significant interest, with applications to many areas, including algebraic geometry \cite{MSGeometry}, optimisation \cite{BMaxLinear}, and combinatorics \cite{JCombinatorics}.
In recent years, there has been much interest in identifying which semigroups can be faithfully represented by matrices over the tropical semiring, with a particular focus on semigroups that can be represented by the submonoid of upper triangular matrices \cite{ARStylic,CJKMPlacticLikeRep,JKPlacticRep}. In addition to this, there has also been a concerted effort to further understand the semigroup theoretic properties of subsemigroups of the monoid of $n \times n$ matrices over the tropical semiring, $M_n(\T)$, \cite{GJNMatrixOverSemirings, GIMMUltimate,IJKDimPolytopes,STropical}.

\par Initially, the focus was on classifying the Green's relations of $M_n(\T)$ and its subsemigroups \cite{HKDuality,JKGreensJ}; for example, up to isomorphism, the maximal subgroups of $M_n(\F\T)$ were fully classified in \cite{IJKTropicalGroups}, where $\F\T = \T \setminus \{-\infty\}$ is the \emph{finitary tropical semiring}. However, there is a flaw in the proof of this result that is explained and corrected (via Theorem~\ref{thm:MaximalSubgroupsClassification}) in Section~\ref{sec:GroupCase}.

\par Rather than working with the monoid $M_n(\T)$ of $n \times n$ matrices, it is sometimes useful to study $M(\T)$, the category of all matrices of any size over $\T$, where multiplication is defined for matrices with compatible sizes.  
\par In this paper, we study the structure of $M(\T)$ as a category, focusing on the Sch\"utzenberger groups and the maximal subgroups of $M(\T)$. These groups are crucial to understanding the structure of semigroups, with Sch\"utzenberger groups acting as analogues of maximal subgroups in non-regular $\cD$-classes. 
The two main results of this paper are a complete classification of the Sch\"utzenberger groups and the maximal subgroups of $M(\T)$. From these classifications, we obtain the corresponding classifications for the semigroupoid $M(\F\T)$ and for the semigroups $M_n(\T)$ and $M_n(\F\T)$, thereby reproving the main result of \cite{IJKTropicalGroups}.

\par Including this introduction, this paper consists of seven sections. Section~\ref{sec:prelim} gives the required preliminaries, giving a background into tropical mathematics and the semigroupoid analogue to some classical semigroup theory results. In Section~\ref{sec:Schutz}, we prove that each Sch\"utzenberger group of $M(\T)$ is isomorphic to the automorphism group of the column space of any matrix in the corresponding $\cH$-class. 
In Section~\ref{sec:FullRank}, we show that each Sch\"utzenberger group embeds in the group of units of $M_n(\T)$ for some $n \in \N$, and show that we can reduce the problem to the submatrices corresponding to connected components in a bipartite graph formed from the matrix. 
In Section~\ref{sec:Eigenvalues}, we study the eigenvalues of the unit matrices corresponding to elements of Sch\"utzenberger groups to give a more detailed description of the groups, and in Section \ref{sec:SchuzClassification} we give the complete classification of the Sch\"utzenberger groups of $M(\T)$, and $M(\F\T)$. Finally, in Section~\ref{sec:GroupCase}, we discuss two errors in \cite{IJKTropicalGroups} providing a corrected proof of \cite[Theorem 5.10]{IJKTropicalGroups}, and a correction to the remark following Corollary~5.11 in \cite{IJKTropicalGroups}. Moreover, while correcting the proof of \cite[Theorem 5.10]{IJKTropicalGroups}, we generalise the result by completely classifying the maximal subgroup of $M_n(\T)$.

\par \textbf{Acknowledgements.} This work was supported by the Heilbronn Institute for Mathematical Research. The author thanks Mahah Javed for their help with \textsc{Magma} \cite{Magma}. 

\section{Preliminaries} \label{sec:prelim}
Let $\N$ denote the set of positive integers. For a finite set $\Omega$, let $\cS_\Omega$ denote the symmetric group on $\Omega$; we write $\cS_n$ when $\Omega = \{1,\dots,n\}$.
\subsection{Categories and semigroupoids}
Throughout, we refer to categories by their collection of arrows, that is, we denote a category with objects $S_0$ and arrows $S$, by $S$. 
A \emph{semigroupoid} is a category where it is not required that every object has an identity map, and we similarly refer to them by their collection of arrows.
\par Although, in this paper, we mainly work with the category $M(\T)$, we sometimes discuss the semigroupoid $M(\F\T)$, as such, we present the results in this section in the more general case of semigroupoids.

\par Given a semigroupoid $S$, let $S^1$ be the category obtained by adjoining an identity arrow to each object which does not already have an identity arrow. Remark that $S^1 = S$ for any category $S$. We can now define Green's relations for semigroupoids.
\begin{definition}
Let $S$ be a semigroupoid. For $a,b \in S$, we say that:
\begin{enumerate}[(i)]
\item $a \cR b$ if $aS^1 = bS^1$,
\item $a \cL b$ if $S^1a = S^1b$,
\item $a \cJ b$ if $S^1aS^1 = S^1bS^1$,
\item $a \cH b$ if $a \cL b$ and $a \cR b$,
\item $a \cD b$ if $a \cL z \cR b$ for some $z \in S^1$.
\end{enumerate}
\end{definition}
For $a \in S$, we let $H_a$ denote the $\cH$-class containing $a$.

\begin{definition}
Let $H$ be a $\cH$-class of a semigroupoid $S$. Suppose $s \in S^1$ such that $sh$ is defined for all $h \in H$ and $sH \subseteq H$, then, we define
\[ \phi_s: H \rightarrow H, \ a \mapsto sa 
\text{ and } \Gamma_H = \{\phi_s \ | \ s \in S^1, sH \subseteq H \}, \]
and call $\Gamma_H$ the \emph{(left) Sch\"utzenberger group of} $H$. Similarly, we define the \emph{right Sch\"utzenberger group of} $H$, ${}_H\Gamma$, to be
\[ {}_H\Gamma = \{\varphi_s: H \rightarrow H, \ a \mapsto as \ | \ s \in S^1, Hs \subseteq H \}. \]
\end{definition}

\begin{proposition} \label{SchutzProps}
    Let $H$ be a $\cH$-class of a semigroupoid $S$. Then, the Sch\"utzenberger group $\Gamma_H$
    \begin{enumerate}[(i)]
        \item is a group acting on $H$ by permutations;
        \item acts simply transitively on $H$, that is, for all $x,y \in H$ there exists a unique $g \in \Gamma_H$ with $gx = y$;
        \item is isomorphic to ${}_H\Gamma$ as an abstract groups;
        \item is isomorphic to $\Gamma_{H'}$ as a permutation group if $H'$ is a $\cH$-class in the same $\cD$-class as $H$;
        \item is isomorphic to $H$ as an abstract group if $H$ is a maximal subgroup of $S$.
    \end{enumerate}
\end{proposition}
\begin{proof}
    Let $(T,\cdot)$ be the semigroup with underlying set $S \cup \{0\}$ where 0 is a multiplicative zero and for $x,y \in S$, $x \cdot y = xy$ if defined in $S$, and equal to 0 otherwise. 
    Note that $T$ has the same greens relations as $S$ apart from the addition of a $\cJ$-class containing only $0$, so as the proposition holds for semigroups by \cite[Section 2.4]{CPSemigroups}, we are done.
\end{proof}
Note that the dual to the above proposition holds for right Sch\"utzenberger groups.

\subsection{Tropical mathematics}
\par Recall that $\T = (\R \cup \{-\infty\}, \oplus, \otimes)$ is called the \emph{tropical semiring}, where $\oplus$ denotes maximum and $\otimes$ denotes addition. For $n \in \N$, we call $\T^n$ \emph{(affine) tropical $n$-space}, that is, the space consisting of $n$-tuples $x$ of $\T$; we write $x_i$ for the $i$th component of $x$. We define addition and scaling for $\T^n$ by $(x \oplus y)_i = x_i \oplus y_i$ and $(\lambda \otimes x)_i = \lambda \otimes x_i$ respectively. This gives $\T^n$ the structure of a $\T$-semimodule.

\par Let $M(\T)$ denote the set of all matrices of any size with entries in $\T$, and $M_{r \times c}(\T) \subseteq M(\T)$ denote the subset of matrices with $r$ rows and $c$ columns.
The operations $\oplus$ and $\otimes$ from $\T$ induce operations on $M(\T)$ when the addition and the product of two matrices are defined respectively. This gives $M(\T)$ the structure of the arrows of a category, where the objects are given by $\T^n$ for all $n \in \N$. We let $M_n(\T) \subseteq M(\T)$ denote the monoid of $n \times n$ matrices over $\T$.
Note that, there is a natural left action of $M(\T)$ on $\T^n$ by treating elements of $\T^n$ as column vectors and acting elements of $M(\T)$ by tropical matrix multiplication if defined. Similarly, there is a right action of $M(\T)$ on $\T^n$ by treating $\T^n$ as row vectors.

\par Subsemimodules of $\T^n$, that is, subsets closed under tropical addition and scaling are called \emph{tropical convex sets}, if a tropical convex set is \emph{finitely generated}, then it is called a \emph{tropical polytope}. 
\par For each $A \in M_{n \times m}(\T)$, let $C(A)$ be the \emph{column space} of $A$, that is, the tropical polytope in $\T^n$ generated by the columns of $A$. Similarly, we let $R(A)$ be the \emph{row space} of $A$, the tropical polytope in $\T^m$ generated by the rows of $A$.

\section{Green's relations and Sch\"utzenberger groups} \label{sec:Schutz}
In this section, we show that for each $\cH$-class of $M(\T)$, $H$, the Sch\"utzenberger group of $H$ is isomorphic to the automorphism group of the column space of any matrix in $H$.

\begin{proposition} \label{prop:Greens}
For all $A,B \in M(\T)$:
\begin{enumerate}[(i)]
\item $A \cR B$ if and only if $C(A) = C(B)$;
\item $A \cL B$ if and only if $R(A) = R(B)$;
\item $A \cH B$ if and only if $R(A) = R(B)$ and $C(A) = C(B)$;
\item $A \cD B$ if and only if $C(A)$ and $C(B)$ are isomorphic as $\T$-semimodules; 
\item $A \cD B$ if and only if $R(A)$ and $R(B)$ are isomorphic as $\T$-semimodules.
\end{enumerate}
\end{proposition}
\begin{proof}
\textit{(i--iii)} follow from a simple generalisation of the proof of \cite[Proposition 4.1]{HKDuality} and \textit{(iv)} and \textit{(v)} follow from the proof of \cite[Theorem 6.1]{HKDuality}.
\end{proof}

From the above proposition, we obtain the following two corollaries. 
\begin{corollary} \label{cor:hclasses}
Suppose $A \in M(\T)$ is $\cH$-related to $B \in M_{n \times m}(\T)$, then $A \in M_{n \times m}(\T)$.
\end{corollary}

\begin{corollary} \label{cor:ColSameRowSame} 
Let $A,B \in M(\T)$. Then, $C(A) \cong C(B)$ if and only if $R(A) \cong R(B)$.
\end{corollary}

\begin{theorem} \label{thm:auto}
Let $A \in M_{n \times m}(\T)$, $H = H_A$, and $\mathrm{Aut}(C(A))$ denote the group of $\T$-semimodule automorphisms of $C(A)$. Then
\begin{enumerate}[(i)]
\item for each $f \in \mathrm{Aut}(C(A))$ there exists $F \in M_n(\T)$ such that $\phi_F \in \Gamma_H$ maps $A$ to the matrix obtained by applying $f$ to the columns of $A$, and
\item the map $\psi\colon \mathrm{Aut}(C(A)) \rightarrow \Gamma_H, f \mapsto \phi_F$ is an isomorphism with inverse $\varphi \colon \Gamma_H \rightarrow \mathrm{Aut}(C(A))$, $\phi_X \mapsto (x \mapsto X \otimes x).$
\end{enumerate}
\end{theorem}

\begin{proof}
For $X \in M(\T)$, let $X_i$ denote the $i$th column of $X$. Then, let $f \in \mathrm{Aut}(C(A))$ and $B \in M_{n\times m}(\T)$ be the matrix obtained by applying $f$ to the columns of $A$, that is, $f(A_i) = B_i$.
Since $f$ is a surjective linear map, $C(A) = C(B)$, so $A \cR B$. Note that $f$ is a $\T$-linear isomorphism of column spaces, taking the $i$th column of $A$ to the $i$th column of $B$, so by \cite[Theorem 4.6 \& Corollary 5.3]{HKDuality} $R(A) = R(B)$, and hence $A \cL B$. Thus, $A\cH B$.
\par By Proposition~\ref{SchutzProps}\textit{(ii)}, $\Gamma_H$ acts simply transitively on $H$, so there exists a unique $\phi_F \in \Gamma_H$ such that $\phi_FA = B$. Moreover, $F \in M_n(\T)$ as $A,B \in M_{n \times m}(\T)$, so we have shown \textit{(i)}. 

\par The map $\psi$ is well-defined as $\Gamma_H$ acts simply transitively on $H$. 
Now, let $\psi(f) = \phi_F$ and $x \in C(A)$. Then, $x = \bigoplus_{i=1}^{m} (\lambda_i \otimes A_i)$ for some $\lambda_i \in \T$ and, by the linearity of $f$ and matrix multiplication,

\[f(x) = \bigoplus_{i=1}^{m} (\lambda_i \otimes f(A_i)) = \bigoplus_{i=1}^{m} (\lambda_i \otimes FA_i) = F \otimes x \]
for all $x \in C(A)$.
\par To show that $\psi$ is a homomorphism of groups, let $f,g \in \mathrm{Aut}(C(A))$, $\psi(f) = \phi_F$, and $\psi(g) = \phi_G$ such that $\phi_FA = B$ and $\phi_GB = C$. Then, $f(A_i) = B_i$ and $g(B_i) = C_i$ and $\psi(g \circ f) = \phi_Z$ for some $Z \in M_n(\T)$. Now,
\[(\phi_ZA)_i = ZA_i = (g \circ f)(A_i) = g(B_i) = C_i, \]
and hence $[\psi(g \circ f)](A) = C$. On the other hand,
\[ [\psi(g)\psi(f)](A) = \phi_G(\phi_FA) = \phi_G(B) = C.\]
Therefore, $[\psi(g \circ f)](A) = [\psi(g)\psi(f)](A)$ and, as $\Gamma_H$ acts simply transitively on $H$, $\psi(g \circ f) = \psi(g) \psi(f)$.

\par Let $f \in \mathrm{Aut}(C(X))$ such that $\psi(f) = \phi_{I_n}$. Then, as $\phi_{I_n}$ fixes $A$, $f$ fixes the columns of $A$. Thus, $f$ is a trivial automorphism and hence, $\psi$ has trivial kernel. Therefore, $\psi$ is injective.

\par To see that $\psi$ is surjective, let $\phi_F \in \Gamma_H$ and define $f\colon C(A) \rightarrow C(A)$ by $f(x) = F \otimes x$ for all $x \in C(A)$. Note that $\phi_F$ maps $H$ to $H$, so $f$ maps $C(A)$ to $C(A)$, and hence $F \otimes x \in C(A)$ for all $x \in C(A)$. 

\par Since $\Gamma_H$ is a group, there exists $\phi_{F'} \in \Gamma_H$ such that $\phi_F\phi_{F'} = \phi_{I_n} = \phi_{F'}\phi_F$. Now, define $f' \colon C(A) \rightarrow C(A)$ by $f'(x) = F' \otimes x$ for all $x \in C(A)$. Then, $f(f'(x)) = x = f'(f(x))$, and hence $f$ is a bijection with inverse $f'$. Thus, $f \in \mathrm{Aut}(C(A))$ and $\psi$ is a surjection as $\psi(f) = \phi_F$.
\par Finally, we show that $\varphi \circ \psi$ is the identity map on $\mathrm{Aut}(C(A))$. Let $x \in C(A)$ and $f \in \mathrm{Aut}(C(A))$ with $\psi(f) = \phi_F$. Then,
\[ [\varphi(\psi(f))](x) = [\varphi(\phi_F)](x) = F \otimes x = f(x)\]
Thus, $\varphi(\psi(f)) = f$, and so $\phi$ and $\varphi$ are mutually inverse group isomorphisms.
\end{proof}

By considering automorphisms of row spaces rather than column spaces, the above theorem can easily be adapted to show that $\mathrm{Aut}(R(A))$ is isomorphic to ${}_H\Gamma$, the right Sch\"utzenberger group of $H$.

\begin{corollary} \label{cor:CRSameAut} 
Let $A \in M_{n \times m}(\T)$. Then $\mathrm{Aut}(C(A)) \cong \mathrm{Aut}(R(A))$ as abstract groups.
\end{corollary}
\begin{proof}
By Theorem \ref{thm:auto}, its dual statement for row spaces, and Proposition~\ref{SchutzProps}\textit{(iii)}, we have that, 
\[ \mathrm{Aut}(C(A)) \cong \Gamma_{H_A} \cong {}_{H_A}\Gamma \cong \mathrm{Aut}(R(A))\] 
as abstract groups.
\end{proof}

\section{Group of units embedding} \label{sec:FullRank}
In this section, we introduce the \emph{row and column rank} of a matrix and show that to identify the isomorphism types of Sch\"utzenberger groups of $M(\T)$, it suffices to consider full rank matrices. We then show that, given a full rank matrix $A \in M_{n \times m}(\T)$, that we can embed its Sch\"utzenberger group in the group of units of $M_n(\T)$. Using this embedding, we can realise each maximal subgroup as a direct product of semidirect products.

For $X \in M(\T)$, we define the \emph{column rank} and \emph{row rank} of $X$ to be the minimum cardinality of a generating set, under tropical scaling and addition, of $C(X)$ and $R(X)$ respectively. If $X \in M_{n \times m}(\T)$ has row rank $r$ and column rank $c$, then $0 \leq r \leq n$ and $0 \leq c \leq m$, and we say that $X$ has \emph{full row rank} if $r = n$ and \emph{full column rank} if $c = m$. If $X$ has full row and column rank, we say $X$ has \emph{full rank}.

\begin{lemma} \label{lem:reduced}
    Let $X \in M_{n \times m}(\T)$ have row rank $r$ and column rank $c$. Then, for every $k \geq r$ and $l\geq c$ there exists $Z \in M_{k\times l}(\T)$ such that $C(X) \cong C(Z)$, $R(X) \cong R(Z)$, $\Gamma_{H_X} \cong \Gamma_{H_Z}$ and ${}_{H_X}\Gamma \cong {}_{H_Z}\Gamma$.
\end{lemma}
\begin{proof}
    By \cite[Corollary~20]{BSSGenExtemeBases}, $R(X)$ is generated by $r$ rows of $X$, as $X$ has row rank $r$. 
    Thus, by removing all rows except $r$ which generate $R(X)$, and then duplicating rows until there are $k$ rows, we obtain a $k \times m$ matrix $Y$ with $R(X) = R(Y)$ and, by Corollary~\ref{cor:ColSameRowSame}, $C(X) \cong C(Y)$. 
    Similarly, $Y$ has $c$ columns which generate $C(Y)$, so by removing any columns except $c$ which generate $C(Y)$, and then duplicating columns until there are $l$ columns, we obtain a $k \times l$ matrix $Z$ such that $C(Y) = C(Z)$. Thus, $C(X) \cong C(Z)$, and hence, by Proposition~\ref{prop:Greens}, $X \cD Z$ and $R(X) \cong R(Z)$. Therefore, $\Gamma_{H_X} \cong \Gamma_{H_Z}$ and ${}_{H_X}\Gamma \cong {}_{H_Z}\Gamma$ by Proposition~\ref{SchutzProps}\textit{(iv)} and its dual.
\end{proof}

\par We say that $P \in M(\T)$ is a \emph{unit} if and only if it is $\cH$-related to an identity matrix in $M(\T)$. By Corollary~\ref{cor:hclasses}, all units are square. Furthermore, it is well-known that $P \in M_n(\T)$ is a unit if and only if $P$ is a monomial matrix, that is, if every row and column contains exactly one finite entry. Thus, the unit matrices in $M(\T)$ are precisely the monomial matrices.
\par For a matrix $A \in M_{n\times m}(\T)$, define $G_A$ to be the set of all units $P \in M_n(\T)$ such that $PX \in H_A$ for all $X \in H_A$, or equivalently, that $PA \in H_A$. Note that $R(PA) = R(A)$ for all units $P \in M_n(\T)$, so $G_A$ is exactly the set of units such that $C(PA) = C(A)$ by Proposition~\ref{prop:Greens}.

\begin{theorem} \label{thm:units}
Let $A \in M_{n \times m}(\T)$ have full row rank and $H = H_A$. Then
\[ \gamma \colon G_A \rightarrow \Gamma_H, \ \gamma(P) = \phi_P \]
is a group isomorphism.
\end{theorem}
\begin{proof}
First note that $\phi_P \in \Gamma_H$ for all $P \in G_A$. For $P,Q \in G_A$,
\[\gamma(PQ) = \phi_{PQ} = \phi_P\phi_Q = \gamma(P)\gamma(Q)\]
so, $\gamma$ is a homomorphism of groups.

\par For injectivity, suppose $\phi_P = \phi_{I_n}$ for some $P \in G_A$. 
Then, $PX = X$ for all $X \in H$.
As, $P$ is a monomial matrix, $PX$ is a scaling and permutation of the rows of $X$. However, $A$ has full row rank, so no row of $X$ is a scaling of another row. Thus, $P = I_n$ and hence, $\gamma$ is injective.

\par For surjectivity, let $\phi_F \in \Gamma_H$ and $B = FA \in H$. Since $A \cH B$, and $A$ has full row rank, it follows that the rows of $B$ provide a minimal generating set for the row space of $A$, and hence, by \cite[Corollary~20]{BSSGenExtemeBases}, each row of $B$ is a scaling of a different row in $A$. Thus, there exists a unit $P \in M_n(\T)$ such that $B = PA \in H$, and hence $P \in G_A$. Moreover, as $\Gamma_H$ acts simply transitively on $H$ by Proposition~\ref{SchutzProps}\textit{(ii)}, $\phi_P = \phi_F$.
\end{proof}

For any matrix $A \in M_{n \times m}(\T)$ with no row or column only containing $-\infty$, let $\cB_A$ be the coloured directed bipartite graph with vertex sets $\Omega = \{\omega_1, \dots, \omega_n\}$ and $\Theta = \{\theta_1,\dots,\theta_m\}$, with an edge from $\omega_i$ to $\theta_j$ if and only if $A_{ij} \neq -\infty$ where the edge has colour $A_{i,j}$ if it exists. Note that all edges in $\cB_A$ go from $\Omega$ to $\Theta$. We let $\cC_A$ be the set of connected components of $\cB_A$ and, for $X \in \cC_A$, let $A|_X$ be the submatrix of $A$ containing $A_{ij}$ if and only if $\omega_i,\theta_j \in X$. For convenience, we keep the labelling of the rows and columns of $A$ for $A|_X$.

\begin{example}
    Let $A = \left(\begin{smallmatrix}
        0 & 0 & -\infty & -\infty \\
        -\infty & 1 & -\infty & -\infty\\
        -\infty & -\infty & 1 & 0
    \end{smallmatrix}\right) \in M_{3 \times 4}(\T)$. Then, the coloured directed bipartite graph $\cB_A$ is given by
    
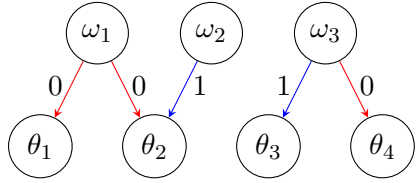
\[\begin{tikzpicture}
       \node[circle,draw] (A) at (0.75,0){$\omega_1$};
        \node[circle,draw] (B) at (2.25,0){$\omega_2$};
        \node[circle,draw] (C) at (3.75,0){$\omega_3$};
        \node[circle,draw] (D) at (0,-1.5){$\theta_1$};
        \node[circle,draw] (E) at (1.5,-1.5){$\theta_2$};
        \node[circle,draw] (F) at (3,-1.5){$\theta_3$};
        \node[circle,draw] (G) at (4.5,-1.5){$\theta_4$};
        \draw[red, -stealth,] (A) to (D);
        \draw[red, -stealth] (A) to (E);
        \draw[blue, -stealth] (B) to (E);
        \draw[blue, -stealth] (C) to (F);
        \draw[red, -stealth] (C) to (G);
        \draw node at (0.2,-0.7) {0};
        \draw node at (1.3,-0.7) {0};
        \draw node at (2.1,-0.7) {1};
        \draw node at (3.2,-0.7) {1};
        \draw node at (4.3,-0.7) {0};
    \end{tikzpicture}\]
    where $\cC_A = \{\{\omega_1,\omega_2,\theta_1,\theta_2\},\{\omega_3,\theta_3,\theta_4\}\}$ and
    \[ A|_{\{\omega_1,\omega_2,\theta_1,\theta_2\}} = \begin{pmatrix}
        0 & 0 \\
        -\infty & 1
    \end{pmatrix} \text{ and } A|_{\{\omega_3,\theta_3,\theta_4\}} = \begin{pmatrix}
        1 & 0
    \end{pmatrix}\]
\end{example}

\begin{remark} \label{rem:reducedrank}
    Let $A \in M_{n \times m}(\T)$ have full row (resp. column) rank. Then, $A|_X$ has full row (resp. column) rank for all $X \in \cC_A$. 
\end{remark}
Note that if $A$ has full rank, then it does not have any row or column only containing $-\infty$, so $\cC_A$ is defined.

Analogously to $G_A$, for $A \in M_{n \times m}(\T)$, let ${}_AG$ be the set of all units $Q \in M_m(\T)$ such that $XQ \in H_A$ for all $X \in H_A$, or equivalently, that $AQ \in H_A$. Note that $C(A) = C(AQ)$ for all units $Q \in M_m(\T)$, so ${}_AG$ is exactly the set of units such that $R(AQ) = R(A)$ by Proposition~\ref{prop:Greens}.

\begin{lemma} \label{dual}
Let $A \in M_{r\times c}(\T)$ have full rank. Then, for all $P \in G_A$ there exists a unique $Q \in {}_AG$ such that $PA = AQ$.
\end{lemma}
\begin{proof}
By Theorems~\ref{thm:auto} and \ref{thm:units}, every $P \in G_A$ corresponds to an automorphism of the column space of $C(A)$. Moreover, as $A$ has full column rank, no column is a multiple of any other column, so $P \in G_A$ only permutes and scales the columns. Hence, $P \otimes A_i = \mu_{\tau^{-1}(i)} + A_{\tau^{-1}(i)}$ for some $\mu_{\tau^{-1}(i)} \in \F\T$ and $\tau \in \cS_c$. 
So, define $Q_{ij} = \mu_i$ if $j = \tau(i)$, and $-\infty$ otherwise. Then, $(AQ)_i = \mu_{\tau^{-1}(i)} + A_{\tau^{-1}(i)}$ and $PA = AQ$. For uniqueness, note that since $A$ has full column rank, if $AQ = AQ'$ for a unit matrix $Q'$ then $Q = Q'$.
\end{proof}

\begin{proposition} \label{prop:semidirectsamecol}
    Let $A \in M_{r \times c}(\T)$ have full rank, $\cC_A = \{C_{1},\dots,C_{n}\}$, and $\Omega_i = \Omega \cap C_i$. Suppose $C(A|_{C_i}) \cong C(A|_{C_j})$ for all $i,j$. Then, let $U^{1,1} = I_{|\Omega_1|}$ and, for each $j \neq 1$, let $U^{1,j}$ be a unit such that $C(U^{1,j}\cdot A|_{C_j}) = C(A|_{C_1})$. Define the groups,
    \begin{align*}
        G_{C_i} &= \{P \in G_A \colon P_{k,k} = 0 \text{ if } \omega_k \notin \Omega_i\} \cong G_{A|_{C_i}} \text{ and} \\
        S_n &= \{ P \in G_A \colon P|_{\Omega_i \times \Omega_j} = \{(-\infty),U^{i,j}\}\} \cong \cS_n,
    \end{align*}
    where $(-\infty)$ denotes a matrix only containing $-\infty$, $U^{i,j} = (U^{1,i})^{-1}U^{1,j}$, and $P|_{\Omega_i \times \Omega_j}$ is the submatrix of $P$ containing $P_{s,t}$ if and only if $\omega_s \in \Omega_i$ and $\omega_t \in \Omega_j$.
    Then, $G_A$ is the internal semidirect product of $\prod_{i=1}^nG_{C_i}$ and $S_n$.
\end{proposition}
\begin{proof} 
\par By Corollary~\ref{cor:ColSameRowSame} and Remark~\ref{rem:reducedrank}, $R(A|_{C_j}) \cong R(A|_{C_1})$ and $A|_{C_j}$ has full rank for all $1 \leq j \leq n$. Thus, for all $j$, $|\Omega_j| = |\Omega_1|$, $|\Theta \cap C_j| = |\Theta \cap C_1|$, and by \cite[Theorem 102]{GLecturesMaxPlus}, there exists a unit $U^{1,j}$ such that $C(U^{1,j}\cdot A|_{C_j}) = C(A|_{C_1})$.

\par Note that if $P \in G_{C_i}$, then $P|_{\Omega_i \times \Omega_i} \in G_{A|_{C_i}}$ and $P|_{\Omega_j \times \Omega_j} = I_{|\Omega_j|}$ for any $j \neq i$. Thus, $G_{C_i} \cong G_{A|_{C_i}}$ for all $i$, and moreover, each $G_{C_i}$ is a subgroup of $G_A$ which intersects trivially and commutes with $G_{C_j}$ for any $j \neq i$. So, $N = \prod_{i=1}^n G_{C_i}$, the internal direct product of all the $G_{C_i}$, is a subgroup of $G_A$.

Let $\sigma, \tau \in \cS_n$ and define units $S,T$ such that $S|_{\Omega_i \times \Omega_{\sigma(i)}} = U^{i,\sigma(i)}$ and $T|_{\Omega_i \times \Omega_{\tau(i)}} = U^{i,\tau(i)}$ for all $i$. Then,
\[(SA)|_{\Omega_i \times \Theta_{\sigma(i)}} = S|_{\Omega_i \times \Omega_{\sigma(i)}}A|_{\Omega_{\sigma(i)} \times \Theta_{\sigma(i)}} = S|_{\Omega_i \times \Omega_{\sigma(i)}}A|_{C_{\sigma(i)}} = U^{i,\sigma(i)}A|_{C_{\sigma(i)}}\]
where $X|_{\Omega_i \times \Theta_j}$ is the submatrix of $X$ containing $X_{s,t}$ if and only if $\omega_s \in \Omega_i$ and $\theta_t \in \Theta_j$.
Thus, $C(SA) = C(A)$ as $C(U^{i,\sigma(i)}A|_{C_{\sigma(i)}}) = C(A|_{C_i})$. Therefore, $S \in G_A$ and hence, $S \in S_n$.
To see that $S^{-1} \in S_n$ note that, for all $i$,
\[ S^{-1}|_{\Omega_{\sigma(i)} \times \Omega_i} = (U^{i,\sigma(i)})^{-1} = ((U^{1,i})^{-1}U^{1,\sigma(i)})^{-1} = (U^{1,\sigma(i)})^{-1}U^{1,i} = U^{\sigma(i),i}.\]
Similarly, to see that $ST \in S_n$ note that, for all $i$ and $j$,
\[(ST)|_{\Omega_i \times \Omega_{\tau(j)}} =
\begin{cases}
    U^{i,\sigma(i)}U^{j,\tau(j)} = U^{i,\tau(j)} &\text{if } j = \sigma(i), \\
    (-\infty) &\text{otherwise.}
\end{cases} \]
Thus, $S_n \cong \cS_n$ and $S_n$ is a subgroup of $G_A$. We now aim to show that $G_A$ is the internal semidirect product of $N$ and $S_n$

First note that $N \cap S_n = I_r$. We now show that $NS_n = G_A$. Let $P \in G_A$ with $P_{i_1,j_1} \neq -\infty \neq P_{i_2,j_2}$. Then, by Lemma~\ref{dual}, there exist $Q \in {}_AG$ such that $PA = AQ$. 

\par Right multiplying by $Q \in {}_AG$ only permutes the columns of $A$, so $\omega_s$ and $\omega_t$ are connected in $\cB_A$ if and only $\omega_s$ and $\omega_t$ are connected in $\cB_{AQ} = \cB_{PA}$. Then, by the definition of $P$, the $j_1$-th and $j_2$-th rows of $A$ are scalings of the $i_1$-th and $i_2$-th rows of $PA$ respectively.
Thus, $\omega_{i_1}$ and $\omega_{i_2}$ are connected in $\cB_A$ if and only if $\omega_{j_1}$ and $\omega_{j_2}$ are connected in $\cB_A$. Therefore, $P$ maps connected components to connected components.

\par Suppose $P \in G_A$, then there exists $\sigma \in \cS_n$ and units $P_i$ such that $P|_{\Omega_i \times \Omega_{\sigma(i)}} = P_i$ for all $i$. Now, let $D$ be the diagonal matrix such that $D|_{\Omega_i \times \Omega_i} = P_iU^{\sigma(i),i}$, and let $S \in S_n$ with $S|_{\Omega_i \times \Omega_{\sigma(i)}} = U^{i,\sigma(i)}$. Then, $P = DS$ as $U^{\sigma(i),i}U^{i,\sigma(i)} = I_{|\Omega_{\sigma(i)}|}$. 
Moreover, $P_i$ maps $C(A|_{C_i})$ to $C(A|_{C_{\sigma(i)}})$, so $P_iU^{\sigma(i),i} \in G_{A|_{C_i}}$ and hence, $D \in N$. Thus, $G_A = NS_n$.

\par Finally, we show $N$ is normal in $G_A$. Let $D \in N$ and $S \in S_n$ with $S|_{\Omega_i \times \Omega_{\sigma(i)}} = U^{i,\sigma(i)}$ for some $\sigma \in \cS_n$. It suffices to show $SDS^{-1}$ maps each connected component to itself. So, for all $i$,
\[(SDS^{-1})|_{\Omega_i \times \Omega_i} = U^{i,\sigma(i)}D|_{\Omega_{\sigma(i)} \times \Omega_{\sigma(i)}}U^{\sigma(i),i} \neq (-\infty)\] 
as $D|_{\Omega_{\sigma(i)} \times \Omega_{\sigma(i)}} \neq (-\infty)$ for all $i$. Hence, $SDS^{-1} \in N$ and $N$ is normal in $G_A$. Therefore, $G_A$ is the internal semidirect product of $N$ and $S_n$.
\end{proof}

\begin{theorem} \label{embeddable}
Let $A \in M_{r \times c}(\T)$ have full rank and $\cup_{i=1}^k \cC_i$ be a partition of $\cC_A$ such that $X,Y \in \cC_i$ if and only if $C(A|_X) \cong C(A|_Y)$. For each $1 \leq i \leq k$, let $\cC_i = \{C_{i,1},\dots,C_{i,n_i}\}$ and 
\[G_{\cC_i} = \{P \in G_A \colon P_{j,j} = 0 \text{ if } \omega_j \notin \cup_{m=1}^{n_i} C_{i,m} \}.\]
Then, $G_A$ is the internal direct product of $G_{\cC_1}, \dots, G_{\cC_{k-1}}$, and $G_{\cC_k}$, that is, $G_A = \prod_{i=1}^k G_{\cC_i}$.
\end{theorem}
\begin{proof}
\par By the definition of $G_{\cC_i}$, $G_{\cC_i} \cap G_{\cC_j} = \{I_r\}$ and $G_{\cC_i}$ and $G_{\cC_j}$ commute for all $i \neq j$. Thus, $\prod_{i=1}^k G_{\cC_i} \leq G_A$. 

\par Let $P \in G_A$. By a similar argument to the proof of Proposition~\ref{prop:semidirectsamecol}, $P$ maps connected components to connected components. So, suppose $P_{i,j} \neq -\infty$ for some $\omega_i \in C_{s,p}$ and $\omega_j \in C_{t,q}$. Then, as $A$ has full rank, by Theorem~\ref{thm:auto} and \ref{thm:units}, $P$ is an isomorphism of $C(A)$, and hence $C(A|_{C_{s,p}}) \cong C(A|_{C_{t,q}})$. Thus, $s = t$ and $G_A = \prod_{i=1}^k G_{\cC_i}$.
\end{proof}

\section{Eigenvalues} \label{sec:Eigenvalues}
In the previous section, we showed that when $A$ has full rank, $G_A$ is a direct product of semidirect products. We now further investigate the structure of $G_A$ by studying the eigenvalues of its elements when $A$ has full rank and the associated bipartite graph $\cB_A$ is connected.

\par For any square matrix $A \in M_n(\T)$, we can construct a weighted directed graph $\cG_A$ on $n$ points, with an edge from $i$ to $j$ if and only if $A_{i,j} \neq -\infty$, where the edge has weight $A_{i,j}$ if it exists. It is well known that $(A^t)_{i,j}$ is the maximum weighted path of length $t$ from $i$ to $j$ in $\cG_M$.
Moreover, when $A \in M_n(\T)$ is a unit, every cycle mean of $A$, the average weight of a path from a node to itself, is an eigenvalue of $A$ \cite[Chapter 4]{BMaxLinear}.

\par We begin by showing that when $A$ has full rank and $\cB_A$ is connected, $G_A$ is isomorphic to $\R \times \Sigma$, where $\Sigma$ is a finite group. To establish this result, we require the following two lemmas concerning the eigenvalues of the elements in $G_A$.

\begin{lemma} \label{lem:1eigen}
Let $A \in M_{r \times c}(\T)$ have full rank with $\cB_A$ connected. Then, each $P \in G_A$ has exactly one eigenvalue.
\end{lemma}
\begin{proof}
Clear for $r = 1$, so let $r > 1$. Suppose for a contradiction that $P$ has multiple distinct eigenvalues. Then, as $P$ is a monomial, $P^k$ is a diagonal matrix for some $k$. Moreover, as $P$ has distinct eigenvalues so does $P^k$, so $P^k$ is not a scalar matrix. 

As $\cB_A$ is connected, there exists a column $X$ of $A$ with $X_i,X_j \neq -\infty$ and $(P^k)_{i,i} \neq (P^k)_{j,j}$. Thus, $C((P^k)^nX) \neq C((P^k)^mX)$ for any $n \neq m$, and, as $A$ has full column rank, $C(A) \neq C(P^{k}A)$. Therefore, $P^k \notin G_A$ and hence $P \notin G_A$.
\end{proof}

\begin{lemma} \label{0cycle}
Let $A \in M_{r\times c}(\T)$ have full rank with $\cB_A$ connected. Suppose $P, Q \in G_A$ have 0 as an eigenvalue. Then, 0 is an eigenvalue of $PQ$.
\end{lemma}
\begin{proof}
By Lemma~\ref{lem:1eigen}, every cycle in $\cG_P$ and $\cG_Q$ has mean weight 0, and hence the (classical) sum of the finite entries in $P$ is 0, and the sum of the finite entries in $Q$ is also 0. 
Thus, the sum of the finite entries in $PQ$ is 0 as it is equal to the sum of the finite entries of $P$ and $Q$.
Further, $PQ \in G_A$ so applying Lemma \ref{lem:1eigen} again yields that every cycle in $\cG_{PQ}$ has the same mean weight. Since the sum of the entries is a weighted sum of cycle means, and the cycle means are all the same, we deduce that the cycle means are all 0. Hence, 0 is an eigenvalue of $PQ$.
\end{proof}

The following proof is similar to that of \cite[Theorem~5.8]{IJKTropicalGroups}, but is given here for completeness.
\begin{theorem} \label{thm:embeddablesingle}
Let $A \in M_{r \times c}(\T)$ have full rank with $\cB_A$ connected. Let $R = \{ \lambda \otimes I_r \colon \lambda \in \F\T\} \cong \R$, and
    \[\Sigma = \{ P \in G_A \colon P \text{ has eigenvalue 0}\}.\]
Then, each $G_A$ is the internal direct product of $R$ and $\Sigma$, where $\Sigma$ is a finite group embeddable in $\cS_r$.
\end{theorem}
\begin{proof}
First note that $R \leq G_A$. For $A, B \in \Sigma$, clearly $A^{-1} \in \Sigma$ and, by Lemma \ref{0cycle}, $AB \in \Sigma$, so $\Sigma \leq G_A$.
By Lemma~\ref{lem:1eigen}, every diagonal matrix in $G_A$ is in $R$, and $I_r$ is the only matrix in $R$ with 0 as its eigenvalue. Thus, $R \cap \Sigma = \{ I_r\}$.

\par Let $P \in G_A$, by Lemma~\ref{lem:1eigen}, $P$ has an exactly one eigenvalue, $\lambda$. Then, we may write $P = (\lambda \otimes I_r)((-\lambda \otimes I_r)P)$, where $\lambda \otimes I_r \in R$ and $(-\lambda \otimes I_r)P \in \Sigma$.
Thus, as this holds for all $P \in G_A$, we get that $G_A = R\Sigma$. Furthermore, every element of $R$ commutes with every element of $\Sigma$, so $G_A$ is the internal direct product of $R$ and $\Sigma$.

\par Let $\theta: G_A \rightarrow \cS_r$ be the map that sends a monomial matrix to its underlying permutation. Then, by the definition of matrix multiplication, $\theta$ is a homomorphism of groups, with kernel $R$. Hence,
\[ \Sigma \cong G_A/R \cong \mathrm{Im}(\theta) \leq \cS_r. \qedhere\]
\end{proof}

Using the above theorem, we now describe the structure of $G_A$ up to isomorphism.
\begin{corollary} \label{cor:embeddableisomorphism}
    Let $A \in M_{r \times c}(\T)$ have full rank, and $\cup_{i=1}^k \cC_i$ be a partition of $\cC_A$ such that $X,Y \in \cC_i$ if and only if $C(A|_X) \cong C(A|_Y)$. For each $1 \leq i \leq k$, let $\cC_i = \{C_{i,1}, C_{i,2},\dots,C_{i,n_i}\}$. Then,
\[ G_A  \cong \prod_{i=1}^k( (\R \times G_i) \wr \cS_{n_i})\]
where $G_i= \{ P \in G_{A|_{C_{i,1}}} \colon P \text{ has the eigenvalue 0}\}$ are finite groups embeddable in the symmetric group $\cS_{r_i}$, where $r_i$ is the number of rows in $C_i$. 
\end{corollary}
\begin{proof}
    By Theorems~\ref{thm:auto} and \ref{thm:units}, $G_{A|_X} \cong G_{A|_Y}$ for any $X,Y \in \cC_i$. Thus, by Proposition~\ref{prop:semidirectsamecol} and Theorems~\ref{embeddable} and \ref{thm:embeddablesingle}, we are done.
\end{proof}

\begin{lemma} \label{lem:eigenforall}
    Let $A \in M_{r \times c}(\T)$ have full rank with $\cB_A$ connected. Then, there exist vectors $u \in \F\T^r$ and $v \in \F\T^c$ such that $u$ is a (right) eigenvector for all $P \in G_A$ and $v$ is a (left) eigenvector for all $Q \in {}_AG$.
\end{lemma}
\begin{proof}
    We show that $u$ exists for $G_A$, with the case for $v$ and ${}_AG$ being dual. By Lemma~\ref{lem:1eigen}, each $P \in G_A$ has exactly one eigenvalue, so define
    \[\Sigma = \{ P \in G_A \colon P \text{ has eigenvalue } 0 \}.\]
    By Theorem~\ref{thm:embeddablesingle}, $\Sigma$ is a subgroup $G_A$ and for each $P \in G_A$ there exists $\lambda \in \F\T$ such that $(\lambda \otimes I_r)P \in \Sigma$.

    \par Let $P,P' \in \Sigma$ with $P_{i,j}, P'_{i,j} \neq -\infty$. Without loss of generality, suppose $P_{i,j} < P'_{i,j}$, then $(P'P^{-1})_{i,i} = P'_{i,j}P^{-1}_{j,i} > 0$. Thus, we obtain a contradiction as $P'P^{-1} \in \Sigma$ has a cycle mean, and hence an eigenvalue, greater than 0. Therefore, for $P,P' \in \Sigma$, if $P_{i,j}, P'_{i,j} \neq -\infty$ for some $i$ and $j$, then $P_{i,j} = P'_{i,j}$.
    \par We define $u$ recursively. Let $u_1 = 0$, then for all $P \in \Sigma$ such that $P_{t,1} \neq -\infty$, let $u_t = P_{t,1}$. Then, let $k$ be the least such that $u_k$ is not yet defined, and let $u_k = 0$, and then for all $P \in \Sigma$ such that $P_{t,k} \neq -\infty$, let $u_t = P_{t,k}$. Repeat this until $u$ is fully defined. By the previous paragraph, $u$ is well-defined.

    \par If $Q \in G_A$, then $Q = (\lambda \otimes I_r)P$ for some $\lambda \in \F\T$ and $P \in \Sigma$. Thus, to show that $u$ is an eigenvector for all $Q \in G_A$, it suffices to show $u$ is an eigenvector for all $P \in \Sigma$.
    So, let $P \in \Sigma$ and $\sigma \in \cS_r$ such that $P_{i,\sigma(i)} \neq -\infty$ for all $i$. Then, 
    \[ (P \otimes u)_i =  P_{i,\sigma(i)} + u_{\sigma(i)} \text{ for all } i.\]
    Let $1 \leq i \leq r$ and $k$ be the least such that there exists an $R \in \Sigma$ with $R_{i,k} \neq -\infty$. Then, $u_i  = R_{i,k}$, by the definition of $u$. Putting this together with the above equality we obtain 
    \begin{align*}
        (P \otimes u)_i &= (R_{i,k} + (R^{-1})_{k,i}) + P_{i,\sigma(i)} + u_{\sigma(i)}, \\
        &= u_i + (R^{-1}P)_{k,\sigma(i)} + u_{\sigma(i)}, \\
        &= u_i - (P^{-1}R)_{\sigma(i),k} + u_{\sigma(i)}, \\
        &= u_i - u_{\sigma(i)} + u_{\sigma(i)}, \\
        &= u_i,
    \end{align*}
    as $R_{i,k} + (R^{-1})_{k,i} = 0$. Thus, $u$ is an eigenvector for all $P \in \Sigma$, and hence for all $Q \in G_A$.
\end{proof}

\begin{corollary} \label{cor:eigenvec0}
    Let $A \in M_{r \times c}(\T)$ have full rank with $\cB_A$ connected. Then, there exists $B \in M_{r \times c}(\T)$ with $B = UAV$ for some diagonal unit matrices $U,V$ such that $(0,\dots, 0)^T$ is a (right) eigenvector for all $P \in G_B$ and $(0, \dots, 0)$ is a (left) eigenvector for all $Q \in {}_BG$.   
\end{corollary}
\begin{proof}
By Lemma~\ref{lem:eigenforall}, there exists $u \in \F\T^r$, a (right) eigenvector for all $P \in G_A$, and $v \in \F\T^c$, a (left) eigenvector for all $Q \in {}_AG$. Let $U \in M_r(\T)$ and $V \in M_c(\T)$ be the diagonal matrices with $-u$ and $-v$ on the diagonal respectively. Define $B = UAV$, and note that as $U$ and $V$ are unit matrices, $B$ has full rank and $C(A) \cong C(B)$ by Proposition~\ref{prop:Greens}\textit{(iv)}.

\par Let $P \in G_A$. By Lemma~\ref{dual}, there exists a unique $Q \in {}_AG$ such that $PA = AQ$. Thus,
\[ UPU^{-1}B = UPAV = UAQV = BV^{-1}QV.  \]
Thus, as $G_A \cong G_B$, $P \in G_A$ if and only if $UPU^{-1} \in G_B$, and similarly, $Q \in {}_AG$ if and only if $V^{-1}QV \in {}_BG$.
\par We now show $\ul{0} = (0,\dots, 0)^T$ is a (right) eigenvector for all $P \in G_B$, with $(0,\dots,0)$ and ${}_BG$ being dual. 
Let $P = URU^{-1} \in G_B$ and $\sigma \in \cS_r$ such that $R_{i,\sigma(i)} \neq -\infty$ for all $i$. Note that $R \in G_A$, and, as $u$ is an eigenvector for $R$, 
$(R\otimes u)_i = R_{i,\sigma(i)} + u_{\sigma(i)} = \lambda + u_i,$
for some $\lambda \in \F\T$. Thus,
\[ (P \otimes \ul{0})_i = (URU^{-1} \otimes \ul{0})_i = - u_i + R_{i,\sigma(i)} + u_{\sigma(i)} = \lambda,\]
and hence, $\ul{0}$ is an eigenvector for all $P \in G_B$.
\end{proof}

\section{Classification of the Sch\"utzenberger groups} \label{sec:SchuzClassification}

In this section, we provide an exact characterisation of the Sch\"utzenberger groups of $M(\T)$. We do this by classifying the Sch\"utzenberger groups of $\cH$-classes corresponding to matrices with $n$ rows and $m$ columns for each $n,m \in \N$. We begin by presenting the following well-known result.

\begin{lemma} \label{lem:trivialautlem}
    For $1 \leq i \leq 3$, let $X_i \in M_{n_i \times m_i}(\T)$ such that $\mathrm{Aut}(C(X_i)) \cong \R$ and $k_i = \min(n_i,m_i)$. Then, 
    \begin{enumerate}[(i)]
        \item $C(X_1) \cong C(X_2)$ if $k_1 = k_2 = 1$, and
        \item $C(X_i) \cong C(X_j)$ for some $i \neq j$ if $k_1,k_2,k_3 \leq 2$.
    \end{enumerate}
\end{lemma}
\begin{proof}
\par Let $A = (0)$, $B = \left(\begin{smallmatrix}
    0 & 0 \\
    -\infty & 0
\end{smallmatrix} \right)$, and $X \in M_{n \times m}(\T)$ have full rank with $\mathrm{Aut}(C(X)) \cong \R$. By \cite[Theorem 4.4]{JK2by2}, $\mathrm{Aut}(C(A)) \cong \R \cong \mathrm{Aut}(C(B))$. Without loss of generality suppose $n \leq m$, and note that, as $X$ has full column rank, it has no columns only containing $-\infty$. Thus, if $n = 1$, then $C(X) = C(A)$. 

\par So, suppose $n = 2$. If there exists $i,j,s,t$  with $j \neq t$ such that $X_{i,j},X_{s,t} = -\infty$, then either the $j$-th column is a scaling of the $t$-th column, or every column is the span of the $j$-th and $t$-th column. In the former case, we get a contradiction as $X$ has full column rank and, in the latter case, we get a contradiction as $C(X) = C(I_2)$ and, by Theorem~\ref{thm:units} and Corollary~\ref{cor:embeddableisomorphism}, $\mathrm{Aut}(C(I_2)) \cong G_{I_2} \cong \R \wr \cS_2$.
\par Thus, there is at most one $-\infty$ entry in $X$. Without loss of generality, suppose no $-\infty$ entries are in the first row. By scaling columns, we may assume $X_{1,j} = 0$ for all $j$; note that this does not change the column space. Let $i,j$ be such that $X_{2,i} = \min_l(X_{2,l})$ and $X_{2,j} = \max_l(X_{2,l})$, then every $x \in C(X)$ is a linear combination of the $i$-th and $j$-th column of $X$, and hence they generate $C(X)$. Thus, $X \in M_2(\T)$ as $X$ has full column rank.
If $X_{2,i} \neq -\infty$, then $X \cD Y$ for some $Y = \left( \begin{smallmatrix}
    0 & y \\
    y & 0
\end{smallmatrix}\right)$ with $y < 0$ as we can obtain $Y$ by scaling and permuting rows and columns of $X$. 
By \cite[Theorem~4.4]{JK2by2}, $\mathrm{Aut}(C(Y)) \cong \cS_2 \times \R$, giving a contradiction as $C(X) \cong C(Y)$.
If $X_{2,i} = -\infty$ then $B$ can be obtained by scaling the second row of $X$. Thus, $X \cD B$ and $C(X) \cong C(B)$.

\par Therefore, if $k_1 = k_2 = 1$, then $C(X_1) \cong C(A) \cong C(X_2)$. If $k_1,k_2,k_3 \leq 2$, then there exists $i \neq j$ such that $X_i$ and $X_j$ has the same row rank $r \in \{1,2\}$. If $r = 1$, then $C(X_i) \cong C(A) \cong C(X_j)$, and if $r = 2$, then $C(X_i) \cong C(B) \cong C(X_j)$.
\end{proof}


We say that a coloured bipartite directed graph $G = (U, V, E)$ is \emph{reducible} if there exists a disconnected node or there exist distinct nodes $x,y \in G$ such that 
\begin{align*}
    \text{for all } u \in U, \ (u,x) \in E &\text{ if and only if } (u,y) \in E, \\
    \text{for all } v \in V, \ (x,v) \in E &\text{ if and only if } (y,v) \in E,
\end{align*}
and, if $(u,x) \in E$ (resp. $(x,v) \in E$), then $(u,x)$ and $(u,y)$ (resp. $(x,v)$ and $(y,v))$ are the same colour. We say $G$ is \emph{irreducible} if it is not reducible.

\par Let $G$ be a finite group with faithful actions on sets $\Omega$ and $\Gamma$. We view $G$ as a subgroup $\cS_\Omega \times \cS_\Gamma$ under the diagonal action of these two actions.
We say that $G$ is \emph{paired 2-closed} on $\Omega$ and $\Gamma$ if $G$ contains every element of $\cS_\Omega \times \cS_\Gamma$ which preserves the orbits of pairs in $\Omega \times \Gamma$. If $|\Omega| = n$ and $|\Gamma| = m$, we say that $G$ is paired 2-closed on $n$ and $m$ points.

\par Equivalently, a group $G$ is paired 2-closed on $\Omega$ and $\Gamma$ if and only if $G$ is the full automorphism group an irreducible coloured bipartite directed graph on $\Omega$ and $\Gamma$ such that all the edges go from $\Omega$ to $\Gamma$.

\par Now, to classify the Sch\"utzenberger groups of $M_n(\T)$, we show that given a group $G$ that is paired 2-closed on $n$ and $m$, then we are able to construct $A \in M_{n \times m}(\F\T)$ with $\Gamma_{H_A} \cong \R \times G$.
\begin{lemma} \label{lem:SchuzOneDirection}
    Let $G$ be a group that is paired 2-closed on $n$ and $m$ points such that either $G$ is non-trivial or $n,m > 2$. Then, there exists $A \in M_{r \times c}(\F\T)$ such that,
    \begin{enumerate}[(i)]
        \item $r \leq n$ and $c \leq m$,
        \item $\Gamma_{H_A} \cong \R \times G$, 
        \item $A$ has full rank, and
        \item no row or column only contains a single repeated entry.
    \end{enumerate}
\end{lemma}
\begin{proof}
\par As $G$ is paired 2-closed on $n$ and $m$ points, let $G = \mathrm{Aut}(B')$ where $B'$ is an irreducible coloured bipartite directed graph on $\Omega'$ and $\Theta'$ such that $|\Omega'| = n$ and $|\Theta'| = m$ with all edges going from $\Omega'$ to $\Theta'$.
We may assume that there is an edge from each $\omega \in \Omega'$ to each $\theta \in \Theta'$ as adding all missing edges and colouring them $\delta$, where $\delta$ is not the colour of any edge in $B'$, does not change the automorphism group of the graph.

\par Let $\Delta_{1},\dots,\Delta_{t}$ be the orbits of $\Omega'$ and $\Delta'_{1}, \dots, \Delta'_{t'}$ be the orbits of $\Theta'$ under the action of $G$ on $B'$. 
Next, further colour each edge, in $B'$, to indicate the orbits of the nodes the edge started and ended at, that is, for $x \in \Delta_{i}$ and $y \in \Delta'_{j}$, if the edge from $x$ to $y$ is coloured by $\gamma$ then instead colour it by $\gamma_{i,j}$. Note that this does not change the automorphism group of $B'$ as automorphisms preserve the orbits of the nodes.

\par For $u,v \in \Omega'$, we say that $u \leq v$ if whenever $(u,s)$ and $(u,t)$ are the same colour for some $s,t \in \Theta'$, we have that $(v,s)$ and $(v,t)$ are the same colour. 
Similarly, for $s,t \in \Theta'$, we say that $s \leq t$ if whenever $(u,s)$ and $(v,s)$ are the same colour for some $u,v \in \Omega'$, we have that $(u,t)$ and $(v,t)$ are the same colour. 

\par Now, if $G$ is non-trivial, remove all nodes of $B'$ which are fixed by $G$, and if there exists $i \neq j$ with $u \in \Delta_i$ and $v \in \Delta_j$ such that $u \leq v$, then remove all nodes in $\Delta_j$, and similarly if there exists $i \neq j$ with $u \in \Delta_i'$ and $v \in \Delta_j'$ such that $u \leq v$, then remove all nodes in $\Delta_j'$. Repeat this until no more nodes are removed.

The result of this process, $B$, is a coloured bipartite graph on $\Omega = \{\omega_{1},\dots,\omega_{r}\}$ and $\Theta = \{\theta_{1},\dots,\theta_{c}\}$ for some $r \leq n$ and $c \leq m$ such that there is an edge from each $\omega \in \Omega$ to each $\theta \in \Theta$. 
Moreover, $B$ is irreducible as if there exists $x,y \in B$ such that  $(v,x)$ and $(v,y)$ (resp. $(x,v)$ and $(y,v)$) are the same colour for all $v$, then $x \leq y$ and $y \leq x$, and hence either $x$ or $y$ would be removed.

Moreover, if $x \leq y$ for $x \in \Delta_i$ and $y \in \Delta_j$, then there exists $f\colon \Delta_i \to \Delta_j$ mapping $g \cdot x$ to $g \cdot y$ for $g \in \mathrm{Aut}(B')$. The map is well-defined and surjective as $x \leq y$. Thus, if $g\cdot x = g' \cdot x$ for some $g,g' \in \mathrm{Aut}(B')$, then $g \cdot y = g' \cdot y$. So, removing $\Delta_j$ does not, up to isomorphism, change the automorphism group of $B'$. Therefore, by applying the same logic for $x \in \Delta_i'$ and $y \in \Delta_j'$, we obtain that $\mathrm{Aut}(B) \cong G$.
Note that neither $\Omega$ or $\Theta$ are empty, and if $G$ is trivial, then $r = n$ and $c = m$.

Let $\Delta_{1},\dots,\Delta_{k}$ be the orbits of $\Omega$ and $\Delta'_{1}, \dots, \Delta'_{k'}$ be the orbits of $\Theta$ under the action of $\mathrm{Aut}(B) \cong G$ on $B$. Without loss of generality suppose $k \leq k'$, we may suppose this by considering the bipartite graph obtained by reversing the arrows and taking the transpose of the final matrix.

\par For each colour in $B$, $\gamma$, define $z_{\gamma} \in \F\T$ to be distinct such that together they form a basis for a free abelian group under addition.
Now, define $\cZ_{i,j}$ be the set of colours in $B$ from $\Delta_i$ to $\Delta_j'$, and $Z_{i,j} = \{ z_\gamma \colon \gamma \in \cZ_{i,j}\}$. Solely to guarantee that we construct a matrix with full rank, we have the following additional conditions on our choices for each $z_{\gamma}$, depending on the number of orbits.

\par If $k = 2$, then for each $\gamma \in \cZ_{1,j}$, let $z_{\gamma} \in [-0.1,0.1]$  such that for all $i < j$ with $|Z_{1,i}|,|Z_{1,j}| > 1$, we have that
\begin{equation}
\label{1stRowInequality} \min\{x-y\colon x \neq y \in Z_{1,i}\} > \max\{x-y\colon x \neq y \in Z_{1,j}\},
\end{equation} 
and for each $\gamma \in \cZ_{2,j}$, then let $z_{\gamma} \in [-j - 0.1, -j + 0.1]$ such that for $i < j$ with $|Z_{2,i}|,|Z_{2,j}| > 1$, we have that
\begin{equation}
\label{2ndRowInequality} \max\{x-y\colon x \neq y \in Z_{2,i}\} < \min\{x-y\colon x \neq y \in Z_{2,j}\}. 
\end{equation} 

\par If $k \neq 2$, then for each $\Delta_j'$ choose a node $\theta_{p_j} \in \Delta_j'$ and let $P_j = \cup_{t=1}^{m_j} Q_{j,t}$ be the partition of $\Delta_1$ such that $\omega_x,\omega_y \in Q_{j,t}$ if and only if $(\omega_x,\theta_{p_j})$ and $(\omega_y,\theta_{p_j})$ are the same colour. Define $N_{j,t} = \sum_{i=1}^{t-1} |Q_{j,i}|$.

\par Then, if $\gamma$ is the colour of the edge $(\omega_x,\theta_{p_j})$ for some $\omega_x \in Q_{j,t}$, let  $z_{\gamma} \in [-N_{j,t} -0.1, -N_{j,t} + 0.1]$.
Note that, all nodes in $\Delta_j'$ have the same colour edges entering them, so we have defined $z_\gamma$ for each $\gamma \in \cZ_{1,j}$.
If $\gamma \in \cZ_{i,j}$ for some $i > 1$, then, if $j = 1$, let $z_{\gamma} \in [-0.1,0.1]$ and, if $j > 1$, let
\[  (-1)^{i+j}(i+j)|\Delta_1| - 0.1 \leq z_{\gamma} \leq (-1)^{i+j}(i+j)|\Delta_1| + 0.1.\]
Now, define $A \in M_{r \times c}(\F\T)$ such that $A_{s,t} = z_{\gamma}$ if the edge $(\omega_{s},\theta_{t})$ in $B$ is coloured $\gamma$. 

Note that $A$ has no row or column containing only a single repeated entry, as the corresponding node in $B$ would be fixed by $G$, as $B$ is irreducible.
Moreover, if $\{\omega_i,\omega_j\} \subseteq \Delta_l$ for some $l$, then the $i$-th and $j$-th row contain the same entries, but if $\{\omega_i,\omega_j\} \not\subseteq \Delta_l$ for any $l$, then they have no entries in common. Dually, if $\{\theta_i,\theta_j\} \subseteq \Delta_{l}'$ for some $l$, then the $i$-th and $j$-th column contain the same entries, but if $\{\theta_i,\theta_j\} \not\subseteq \Delta_{l}'$ for any $l$, then they have no entries in common.

We aim to show that $A$ has full rank. 
Let $A_i$ be the $i$-th column of $A$. For a contradiction, suppose there exists $1 \leq t \leq c$ such that
\[A_t = \bigoplus_{j \neq t} \lambda_j \otimes A_j.\]
for some $\lambda_j \in \T$.
For $1 \leq l \leq k$, let $\omega_{s_l} \in \Delta_l$ such that $A_{s_l,t} \geq A_{i,t}$ for all $\omega_i \in \Delta_l$, and let $1 \leq q_l \leq c$ with $q_l \neq t$ such that $A_{s_l,t} = \lambda_{q_l} + A_{s_l,q_l}$.

Now, if $\{\theta_{q_l},\theta_t\} \subseteq \Delta_j'$ for some $l$ and $j$, then $\lambda_{q_l} \geq 0$ as $A_{s_l,t} \geq A_{i,t}$ for all $\omega_i \in \Delta_l$, and $A_{q_l}$ and $A_{t}$ restricted to $\Delta_l$ contain the same entries, so $A_{s_l,t} \geq A_{s_l,q_l}$. Moreover, $A_{q_l}$ and $A_t$ share the same entries and there are no repeated columns, so there exists $i$ such that $A_{i,q_l} > A_{i,t}$ giving a contradiction as $\lambda_{q_l} + A_{i,q_l} > A_{i,t}$.
So, suppose $\theta_t \in \Delta_u'$ and, for each $l$, suppose $\theta_{q_l} \in \Delta_{v_l}'$ for some $v_l \neq u$. 

For the $k \neq 2$ case, let $v = v_1$, $q = q_1$, and $s = s_1$. By relabelling nodes and changing the sets in $P_u$ and $P_v$, we may assume $\theta_t = \theta_{p_u}$ and $\theta_q = \theta_{p_v}$. 

Note that $\lambda_{q} \geq -0.2$ as $A_{s,t} \geq -0.1$ and $A_{s,q} \leq 0.1$. Next, suppose $N_{v,i} < N_{u,j} < N_{v,i+1}$ for some $i$ and $j$, then $A_t$ restricted to $\Delta_1$ contains $N_{u,j}$ elements greater than $-N_{u,j} + 0.1$ and hence at most $N_{u,j}$ elements greater than $-N_{v,i} -0.9$, as $N_{v,i}$ and $N_{u,j}$ are both integers. However, $A_q$ restricted to $\Delta_1$ contains $N_{v,i+1}$ elements greater than or equal to $-N_{v,i}-0.1$. 
Thus, we get a contradiction as $A_q + \lambda_q$ restricted to $\Delta_1$ has more entries greater than or equal to $-N_{v,i}-0.3$ than $A_t$ restricted to $\Delta_1$. Hence, if $N_{v,i} < N_{u,j}$, then $N_{v,i+1} \leq N_{u,j}$. 

As $A_{i,q} -0.2 \leq A_{i,q} + \lambda_q \leq A_{i,t}$, if $\omega_i \in Q_{u,t_u} \cap Q_{v,t_v}$ then we require that $N_{v,t_v} \leq N_{u,t_u}$. Thus, if $N_{v,j} = N_{u,i}$ and $N_{v,j+h} = N_{u,i+1}$, then $\cup_{b=0}^{h-1} Q_{v,j+b} = Q_{u,i}$. Hence, $\theta_q \leq \theta_t$, giving a contradiction, as we removed all nodes which were greater than any other node.

For the $k = 2$ case, note that $\lambda_{q_1} \geq -0.2$ as $A_{s_1,t},A_{s_1,q_1} \in [-0.1,0.1]$, and similarly $\lambda_{q_2} \geq v_2-u-0.2$ as $A_{s_2,t} \geq - u - 0.1$ and $A_{s_2,q_2} \leq - v_2  + 0.1$.
Furthermore, $A_{s_2,q_1} \geq -v_1 - 0.1$ and $A_{s_2,t} \leq -u + 0.1$ so, when $u > v_1$, we get a contradiction as $A_{s_2,q_1} + \lambda_{q_1} > A_{s_2,t}$. Similarly, $A_{s_1,q_2},A_{s_1,t} \in [-0.1,0.1]$ so, when $v_2 > u$, we get a contradiction as $A_{s_1,q_2} + \lambda_{q_2} > A_{s_1,t}$. Therefore, $v_2 < u < v_1$.

If $|Z_{1,u}| > 1$, then there exists $\omega_i \in \Delta_1$ such that $A_{i,t} < A_{s_1,t}$. Moreover, by (\ref{1stRowInequality}), $A_{s_1,t} - A_{i,t} > A_{s_1,q_1} - A_{i,q_1}$, and hence $A_{i,q_1} > A_{s_1,q_1} - A_{s_1,t} + A_{i,t}$.
Therefore, $\lambda_{q_1} + A_{i,q_1} > A_{i,t}$, as $\lambda_{q_1} + A_{s_1,q_1} = A_{s_1,t}$, giving a contradiction, so we may suppose $|Z_{1,u}| = 1$. Similarly, by (\ref{2ndRowInequality}), $|Z_{2,u}| = 1$.
However, this implies that $\theta_t$ is fixed by the automorphism group of $B$, giving a contradiction. 
Thus, in any case, $A$ has full column rank.

\par Let $R_i$ denote the $i$-th row of $A$.
For a contradiction, suppose there exists $1 \leq s \leq r$ such that
\[R_s = \bigoplus_{i \neq s} \lambda_i \otimes R_i.\]
for some $\lambda_i \in \T$.
For $1 \leq l \leq k'$, let $\theta_{t_l} \in \Delta_l'$ such that $A_{s,t_l} \geq A_{s,j}$ for all $\theta_j \in \Delta_l'$, and let $1 \leq p_l \leq r$ with $p_l \neq s$ such that $A_{s,t_l} = \lambda_{p_l} + A_{p_l,t_l}$.

If $\{\omega_{p_l},\omega_s\} \subseteq \Delta_i$ for some $l$ and $i$, then $\lambda_{p_l} \geq 0$ as $A_{s,t_l} \geq A_{s,j}$ for all $\theta_j \in \Delta_l'$, and $R_{p_l}$ and $R_s$ restricted to $\Delta_l'$ contain the same entries, so $A_{s,t_l} \geq A_{p_l,t_l}$. Moreover, $R_p$ and $R_s$ share the same entries and there are no repeated rows, so there exists $j$ such that $A_{p_l,j} > A_{s,j}$ giving a contradiction as $\lambda_{p_l} + A_{p_l,j} > A_{s,j}$. So, suppose $\{\omega_{p_l},\omega_s\} \not\subseteq \Delta_i$ for any $l$ and $i$, and note that we have shown the $k = 1$ case.

For the $k = 2$ case, if $\omega_s \in \Delta_1$ then $\omega_{p_2} \in \Delta_2$ and hence, $\lambda_{p_2} \geq 1.8$ as $A_{s,t_2} \geq -0.1$ and $A_{p_2,t_2} \leq -1.9$. For $j \in \Delta_1'$, $A_{p_2,j} \geq -1.1$ and $A_{s,j} \leq 0.1$, so we obtain a contradiction as $A_{p_2,j} + \lambda_{p_2} > A_{s,j}$. Similarly, if $\omega_s \in \Delta_2$ then $\omega_{p_1} \in \Delta_1$ and hence, $\lambda_{p_1} \geq -1.2$ as $A_{s,t_1} \geq -1.1$ and $A_{p_1,t_1} \leq 0.1$. For $j \in \Delta_2'$, $A_{p_1,j} \geq -0.1$ and $A_{s,j} \leq -1.9$, so we obtain a contradiction as $A_{p_1,j} + \lambda_{p_1} > A_{s,j}$.

For the $k > 2$ case, note that $\lambda_{p_1} \geq 0.8 - |\Delta_1|$ as $A_{s,t_1},A_{p_1,t_1} \in [0.9-|\Delta_1|,0.1]$. Then, as $\{\omega_{p_1},\omega_s\} \not\subseteq \Delta_i$ for any $i$, there exists $\theta_j \in \Delta_2' \cup \Delta_3'$ such that $A_{p_1,j} > A_{s,j}  + |\Delta_1| - 0.2$ which implies that $A_{p_1,j} + \lambda_{p_1} > A_{s,j}$ giving a contradiction. Thus, $A$ has full rank.
We now aim to show that $G_A \cong G \times \R$. 

\par If $r,c \leq 2$, then $G \cong \cS_2$ and $A = \left(\begin{smallmatrix}
        a & b \\
        b & a
    \end{smallmatrix}\right)$ for some $a,b \in \F\T$ with $a \neq b$. Then, $G_A \cong G \times \R$ by \cite[Theorem~4.4]{JK2by2} as $A \cH \left(\begin{smallmatrix}
        0 & -|b-a| \\
        -|b-a| & 0
    \end{smallmatrix}\right)$.
    So, suppose that $\max(r,c) > 2$.
    
    \par We now show that $G_A$ consists of exactly the scalings of permutation matrices corresponding to permutations of $\Omega$ under the action of $G$ on $B$.

    \par For $\nu \in \cS_\Omega \times \cS_\Theta$, we say $\nu = (\sigma,\tau)$ for $\sigma \in \cS_{r}$ and $\tau \in \cS_{c}$ if $\nu$ maps $\omega_i$ to $\omega_{\sigma(i)}$ and $\theta_j$ to $\theta_{\tau(j)}$.
    Now, let $P$ and $Q$ be the permutation matrices such that $P_{i,\sigma(i)} = 0$ and $Q_{i,\tau(i)} = 0$ for all $i$. 
    By the definition of $A$ and $B$, if $\nu = (\sigma,\tau) \in G$, then $PAQ^{-1} = A$, or equivalently, $PA = AQ$.
    Thus, if $\nu \in G$, then $(\lambda \otimes P)A \in H_A$, and hence $\lambda \otimes P \in G_A$ for all $\lambda \in \F\T$.

    Suppose $P \in G_A$ and $\sigma \in \cS_r$ such that $P_{i,\sigma(i)} = \lambda_i \in \F\T$ for all $1 \leq i \leq r$.
    Then, there exists $Q \in {}_AG$ and $\tau \in \cS_c$ such that $PA = AQ$ and $Q_{j,\tau(j)} = \mu_j \in \F\T$ for all $1 \leq j \leq c$.
    Thus, for all $i,j$,
    \begin{equation} \label{eq:entriesofunits}
        \lambda_i + A_{\sigma(i),\tau(j)} = (PA)_{i,\tau(j)} = (AQ)_{i,\tau(j)} = A_{i,j} + \mu_j.
    \end{equation}
    \par Let $\nu = (\sigma,\tau)$. If $\nu \in G$, then the edges $(\omega_i,\theta_j)$ and $(\omega_{\sigma(i)}, \theta_{\tau(j)})$ have the same colour, and hence $A_{i,j} = A_{\sigma(i),\tau(j)}$ for all $i$ and $j$. 
    Therefore, the above equation gives that $\lambda_i = \mu_j$ for all $i$ and $j$, so $P$ is a scaling of the permutation matrix corresponding to $\sigma$.
    If $\nu \notin G$, then there exist $\omega_s \in \Omega$ and $\theta_t \in \Theta$ such that the edges $(\omega_s,\theta_t)$ and $(\omega_{\sigma(s)},\theta_{\tau(t)})$ are different colours in $B$. Thus, by the definition of $A$, $A_{s,t} \neq A_{\sigma(s),\tau(t)}$, so let $A_{s,t} = a$ and $A_{\sigma(s),\tau(t)} = b$.

    \par Suppose $\{\omega_s,\omega_{\sigma(s)}\} \subseteq \Delta_l$ for some $l$. Then, the rows $s$ and $\sigma(s)$ contain the same entries.
    Thus, as $A_{\sigma(s),\tau(t)} = b$ but $A_{s,t} \neq b$, there exists $1 \leq q \leq c$ such that $A_{s,q} = b$ and $A_{\sigma(s),\tau(q)} \neq b$. By equation (\ref{eq:entriesofunits}), $\lambda_i = A_{i,t} - A_{\sigma(i),\tau(t)} + \mu_t$ for all $i$. So, again by equation (\ref{eq:entriesofunits}), when $j = q$,
    \begin{align*}
        (A_{i,t} - A_{\sigma(i),\tau(t)} + \mu_t) +  A_{\sigma(i),\tau(q)} &= A_{i,q} + \mu_q &&\text{for all } i \text{, and } \\
        (a - b + \mu_t) + A_{\sigma(s),\tau(q)} &= b + \mu_q,
    \end{align*}
    where the second equation is obtained when $i = s$.
    Putting these equations together, by equating $\mu_q$, we obtain that
    \[ a +  A_{\sigma(s),\tau(q)} + A_{\sigma(i),\tau(t)} + A_{i,q} = b + b + A_{i,t} + A_{\sigma(i),\tau(q)} \text{ for all } i. \]
    Therefore, as $A_{\sigma(s),\tau(q)} \neq b \neq a$ and the entries of $A$ form a basis for a free abelian group, we get that $A_{\sigma(i),\tau(t)} = b$ for all $i$, and hence the $\tau(t)$ column only contains $b$, giving a contradiction as $A$ contains no columns only containing a single repeated entry.

    \par Similarly, if $\{\theta_t,\theta_{\tau(t)}\} \subseteq \Delta'_{l'}$ for some $1 \leq l' \leq k'$, then this implies that the $\sigma(s)$ row only contains $b$, again giving a contradiction. 
    
    Thus, $\{\omega_s,\omega_{\sigma(s)}\} \not\subseteq \Delta_l$ and $\{\theta_t,\theta_{\tau(t)}\} \not\subseteq \Delta'_{l'}$ for any $1 \leq l \leq k$ and $1 \leq l' \leq k'$, and hence $d := A_{\sigma(s),t} \neq a,b$ and $e := A_{s,\tau(t)} \neq a,b,d$.
    Then, by (\ref{eq:entriesofunits}), when $i \in \{\sigma(s),s\}$ and $j = \tau(t)$, 
    \begin{align*}
    (d - A_{\sigma^2(s),\tau(t)} + \mu_t) +  A_{\sigma^2(s),\tau^2(t)} &= b + \mu_{\tau(t)} \text{, and} \\
    (a - b + \mu_t) + A_{\sigma(s),\tau^2(t)} &= e + \mu_{\tau(t)}.
    \end{align*}
    as $\lambda_i = A_{i,t} - A_{\sigma(i),\tau(t)} + \mu_t$ for all $i$ by (\ref{eq:entriesofunits}). Putting these equations together, by equating $\mu_{\tau(t)}$, we obtain that
    \[ a + A_{\sigma(s),\tau^2(t)} + A_{\sigma^2(s),\tau(t)} = d + A_{\sigma^2(s),\tau^2(t)} + e. \]
    As $\theta_t$ and $\theta_{\tau(t)}$ are in different $\Delta'$, we have that $d \neq A_{\sigma^2(s),\tau(t)}$. Thus, $d = A_{\sigma(s),\tau^2(t)}$ as the entries of $A$ form a basis of a free abelian group. Hence, $\mu_{\tau(t)} = a - b + \mu_t + d - e$. Then, by equation (\ref{eq:entriesofunits}), when $j = \tau(t)$, we get that, for all $i$, 
    \[ (A_{i,t} - A_{\sigma(i),\tau(t)} + \mu_t) + A_{\sigma(i),\tau^2(t)} = A_{i,\tau(t)} + (a - b + \mu_t + d - e). \]
    Rearranging, we obtain that, for all $i$,
    \[ A_{i,t} + A_{\sigma(i),\tau^2(t)} + b + e = a + d + A_{\sigma(i),\tau(t)} + A_{i,\tau(t)}.\]
    Thus, as the entries of $A$ are a basis for a free abelian group, 
    \[ \{A_{i,t},A_{\sigma(i),\tau^2(t)}\} = \{a,d\} \text{ and }
    \{A_{i,\tau(t)},A_{\sigma(i),\tau(t)}\} = \{b,e\} \text{ for all }i.\] However, since $\omega_s$ and $\omega_{\sigma(s)}$ are in different $\Delta$, $a$ and $d$ do not appear in the same row, and neither do $b$ and $e$. Thus, $A_{i,t} = A_{i,\tau^2(t)}$ for all $i$ and, as $B$ is irreducible, $t = \tau^2(t)$. Now, a dual argument shows that
    \[ \{A_{s,j},A_{\sigma^2(s),\tau(j)}\} = \{a,e\} \text{ and } \{A_{\sigma(s),j},A_{\sigma(s),\tau(j)}\} = \{b,d\} \text{ for all }j. \]
    and hence that $\sigma^2(s) = s$.

    \par Then, as each row either contains an $a$ or $d$ and each column either contains an $a$ or $e$, we obtain that $k = k' = 2$, that is, under the action of $G$ on $B$, the orbits of $\Omega$ are $\Delta_1$ and $\Delta_2$, and the orbits of $\Theta$ are $\Delta'_1$ and $\Delta'_2$.
    Moreover, as we know the content of the $s$ and $\sigma(s)$ row and the $t$ and $\tau(t)$ columns, we can see that each $A_{i,j}$ is exactly determined by which orbits $\omega_i$ and $\theta_j$ are contained in. 
    Thus, if any orbit is not a singleton, then $A$ contains either a repeated row or column. Hence, as $\max(r,c) > 2$, we obtain a contradiction.
    
    Therefore, $P \in G_A$ if and only if $P$ is a scaling of a tropical permutation matrix corresponding to a permutation of $\Omega$ under the action of $G$ on $B$, and hence $G_A \cong G \times \R$.
\end{proof}

With the above lemma, we are now able to exactly classify the groups which, up to isomorphism, arise as a Sch\"utzenberger group of $M_{n \times m}(\T)$.

\begin{theorem} \label{thm:SchuzClassification}
The Sch\"utzenberger groups of $M_{n \times m}(\T)$ are, up to isomorphism, exactly groups of the form $\prod_{\alpha=1}^z((\R \times G_\alpha) \wr \cS_{h_\alpha})$ where, for each $1 \leq \alpha \leq z$, there exists $n_\alpha,m_\alpha \in \N$ such that
\begin{enumerate}[(i)]
    \item $\sum_{\alpha=1}^z n_{\alpha}h_{\alpha} \leq n$,
    \item $\sum_{{\alpha}=1}^z m_{\alpha}h_{\alpha} \leq m$,
    \item each $G_{\alpha}$ is paired 2-closed on $n_{\alpha}$ and $m_{\alpha}$ points,
    \item at most one $\alpha$ has $\min(n_\alpha,m_\alpha) = 1$, and
    \item if $G_{\alpha_1},G_{\alpha_2},G_{\alpha_3}$ are trivial for distinct $\alpha_i$, then $\min(n_{\alpha_i},m_{\alpha_i}) > 2$ for some $i$.
\end{enumerate}
\end{theorem}
\begin{proof}
    As left and right Sch\"utzenberger groups are isomorphic, it suffices to show the result for left Sch\"utzenberger groups. Let $X \in M_{n \times m}(\T)$ with row rank $r$ and column rank $c$. By Lemma~\ref{lem:reduced}, there exists $A \in M_{r\times c}(\T)$ with full rank such that $\Gamma_{H_X} \cong \Gamma_{H_A}$, and, by Theorem~\ref{thm:units}, $\Gamma_{H_A} \cong G_A$.

\par Consider the coloured bipartite graph $\cB_A$, and let $\cup_{\alpha=1}^z \cC_\alpha$ be a partition of $\cC_A$ such that $X,Y \in \cC_\alpha$ if and only if $C(A|_X) \cong C(A|_Y)$. For each $1 \leq \alpha \leq z$, define $\cC_\alpha = \{C_{\alpha,1}, C_{\alpha,2},\dots,C_{\alpha,h_\alpha}\}$, 
\begin{align*}
G_\alpha &= \{ P \in G_{A|_{C_{\alpha,1}}} \colon P \text{ has eigenvalue 0 }\} \leq G_{A|_{C_{\alpha,1}}}, \text{ and } \\
{}_\alpha G &= \{ Q \in {}_{A|_{C_{\alpha,1}}}G \colon Q \text{ has eigenvalue 0}\} \leq {}_{A|_{C_{\alpha,1}}}G.
\end{align*}
Then, by Corollary~\ref{cor:embeddableisomorphism},
\[G_A \cong \prod_{{\alpha}=1}^z( (\R \times G_\alpha) \wr \cS_{h_{\alpha}}). \]
\par Without loss of generality, we may suppose $(0,\dots,0)^T$ is a common eigenvector for all $P \in G_{\alpha}$, and $(0,\dots,0)$ is a common eigenvector for all $Q \in {}_\alpha G$ as, by Corollary~\ref{cor:eigenvec0}, for each $C_{\alpha,1}$, there exist diagonal unit matrices $U_\alpha$ and $V_\alpha$ such that  $(0,\dots,0)^T$ is a common right eigenvector for all $G_{U_\alpha A|_{C_{\alpha,1}} V_\alpha}$ and $(0,\dots,0)$ is a common left eigenvector for ${}_{U_\alpha A|_{C_{\alpha,1}} V_\alpha} G$. So by replacing each $A|_{C_{\alpha,1}}$ in $A$ with $U_\alpha A|_{C_{\alpha,1}} V_\alpha$, we obtain a matrix with isomorphic column space with the desired property.

\par Recall $\Omega$ and $\Theta$ are the vertex sets of $\cB_A$. Now, let $\Omega_{C_{\alpha,1}} = \Omega \cap C_{\alpha,1}$ with $n_\alpha = |\Omega_{C_{\alpha,1}}|$ and $\Theta_{C_{\alpha,1}} = \Theta \cap C_{\alpha,1}$ with $m_\alpha = |\Theta_{C_{\alpha,1}}|$. Then, the matrices in $G_\alpha$ are monomial matrices which fix the eigenvector of length $n_{\alpha}$, $(0,\dots,0)^T$, and hence are tropical permutation matrices, inducing a permutation on $n_{\alpha}$ points. Similarly, matrices in ${}_\alpha G$ are tropical permutation matrices that induce a permutation on $m_\alpha$ points.

\par Note that $\cB_{A|_{C_{\alpha,1}}}$ is the subgraph of $\cB_A$ only containing the connected component $C_{\alpha,1}$. For $\nu \in \cS_{\Omega_{C_{\alpha,1}}} \times \cS_{\Theta_{C_{\alpha,1}}}$, we say $\nu = (\sigma,\tau)$ for $\sigma \in \cS_{n_{\alpha}}$ and $\tau \in \cS_{m_{\alpha}}$ if $\nu$ maps $\omega_i$ to $\omega_{\sigma(i)}$ and $\theta_i$ to $\theta_{\tau(i)}$.
Now, let $P$ and $Q$ be the permutation matrix of $\sigma$ and $\tau$ respectively, that is, units such that $P_{i,\sigma(i)} = 0$ and $Q_{i,\tau(i)} = 0$ for all $i$. Then,
\begin{align*}
(\sigma,\tau) \in \mathrm{Aut}(\cB_{A|_{C_{\alpha,1}}}) &\Leftrightarrow (A|_{C_{\alpha,1}})_{\sigma(s),\tau(t)} = (A|_{C_{\alpha,1}})_{s,t} \text{ for all } s \text{ and } t \\
&\Leftrightarrow PA|_{C_{\alpha,1}}Q^{-1} = A|_{C_{\alpha,1}}. \\
&\Leftrightarrow PA|_{C_{\alpha,1}} = A|_{C_{\alpha,1}}Q.
\end{align*}
Thus, $\mathrm{Aut}(\cB_{A|_{C_{\alpha,1}}})$ can be expressed as a diagonal action of an action of $G_\alpha$ on $\Omega_{C_{\alpha,1}}$ and an action of ${}_{\alpha}G \cong G_\alpha$ on $\Theta_{C_{\alpha,1}}$. Hence, $G_\alpha$ has a paired 2-closed representation on $n_{\alpha}$ and $m_{\alpha}$ points. Therefore,
\[\Gamma_A \cong \prod_{{\alpha}=1}^z( (\R \times G_\alpha) \wr \cS_{h_{\alpha}})\]
where each $G_\alpha$ has a paired 2-closed representation on $n_{\alpha}$ and $m_{\alpha}$ points. 
Note that, $\sum_{{\alpha}=1}^z n_{\alpha}h_{\alpha} = r \leq n$ and $\sum_{{\alpha}=1}^z m_{\alpha}h_{\alpha} = c \leq m$. By Lemma~\ref{lem:trivialautlem}, at most one $\alpha$ has $\min(n_\alpha,m_\alpha) = 1$ and if $G_{\alpha_1},G_{\alpha_2},G_{\alpha_3}$ are trivial for distinct $\alpha_i$, then $\min(n_{\alpha_i},m_{\alpha_i}) > 2$ for some $i$.
Thus, we have shown that all Sch\"utzenberger groups are of the given form.

\par Now, let $z \in \N$, and for each $1 \leq \alpha \leq z$, let $n_\alpha,m_\alpha,h_\alpha \in \N$ be such that $\sum_{{\alpha}=1}^z n_{\alpha}h_{\alpha} \leq n$ and $\sum_{{\alpha}=1}^z m_{\alpha}h_{\alpha} \leq m$. Then, let each $G_\alpha$ be a group with paired 2-closed permutation group on $n_{\alpha}$ and $m_{\alpha}$ points such that at most one $\alpha$ has $\min(n_\alpha,m_\alpha) = 1$ and if $G_{\alpha_1},G_{\alpha_2},G_{\alpha_3}$ are trivial for distinct $\alpha_i$, then $\min(n_{\alpha_i},m_{\alpha_i}) > 2$ for some $i$.

By Lemma~\ref{lem:SchuzOneDirection}, for each $\alpha$ with $G_\alpha$ non-trivial or $n_\alpha,m_\alpha > 2$, there exists $A_\alpha \in M_{n_\alpha \times m_\alpha}(\F\T)$ such that $A_\alpha$ has full rank and $G_{A_\alpha} \cong G_\alpha \times \R$. Note that, as $A_\alpha \in M_{n_\alpha \times m_\alpha}(\F\T)$, each $\cB_{A_\alpha}$ is connected, and by the proof of the lemma, we can choose the entries so the set of entries in all the $A_\alpha$ form a basis for a free abelian group.

If there exists $i$ such that $G_{i}$ is trivial with $\min(n_{i},m_{i}) = 2$, let $A_{i} = \left( \begin{smallmatrix}
    0 & 0 \\
    -\infty & 0
\end{smallmatrix}\right)\in M_2(\T)$, and if there exists $j \neq i$ (or no such $i$ existed) such that $G_{j}$ is trivial with $\min(n_{j},m_{j}) \leq 2$, let $A_{j} = (0) \in M_1(\T)$. 
Note that $A_i$ and $A_j$ (if they exist) have full rank.

Let $A \in M_{r \times c}(\T)$ be a block diagonal matrix containing $h_\alpha$ copies of $A_\alpha$ for each $\alpha$. Remark that we allow the blocks of the $A$ to be rectangular, and hence, for $A$ to be rectangular.

Note that, by Corollary~\ref{cor:embeddableisomorphism}, to show $\Gamma_{H_A}$ is isomorphic to $\prod_{{\alpha}=1}^z ((\R \times G_{\alpha}) \wr \cS_{h_{\alpha}})$ it suffices to show that $C(A_\alpha) \not \cong C(A_\beta)$ for any $\alpha \neq \beta$.

    Now, if $G_{\beta}$ is trivial and $\min(n_\beta,m_\beta) \leq 2$, then $C(A_{\alpha}) \not\cong C(A_{\beta})$ for any $\alpha \neq \beta$. So suppose otherwise for both $\alpha \neq \beta$, then $A_{\alpha}$ and $A_{\beta}$ have at least two rows and columns. 
    For a contradiction, suppose $C(A_\alpha) \cong C(A_\beta)$ then, by \cite[Theorem 102]{GLecturesMaxPlus}, there exist units $P,Q$ and permutations $\sigma$ and $\tau$ such that $P_{i,\sigma(i)}, Q_{i,\tau(i)} \neq -\infty$ for all $i$ and $PA_\alpha = A_\beta Q$.
    Then, letting $(A_\alpha)_{i,j} = \alpha_{i,j}$ and $(A_\beta)_{i,j} = \beta_{i,j}$ note that
    \[ P_{i,\sigma(i)} + \alpha_{\sigma(i),\tau(j)} = (PA_\alpha)_{i,\tau(j)} = (A_\beta Q)_{i,\tau(j)} = \beta_{i,j} + Q_{j,\tau(j)}.\]
    for all $i,j$. No column of $A_\beta$ only contains a single value, so there exists $k$ such that $\beta_{1,1} \neq \beta_{k,i}$. Thus, from the above equation, we obtain that,
    \[ \alpha_{\sigma(1),\tau(1)} - \alpha_{\sigma(1),\tau(j)} - \alpha_{\sigma(k),\tau(1)} + \alpha_{\sigma(k),\tau(j)} = \beta_{1,1} - \beta_{1,j} - \beta_{k,1} + \beta_{k,j}\]
    for all $j$. As the entries form a basis for a free abelian group and $A_\alpha$ and $A_\beta$ share no entries, the above equation must be 0 for all $j$. Thus, $\beta_{1,1} = \beta_{1,j}$ for all $j$ as $\beta_{1,1} \neq \beta_{k,1}$, giving a contradiction as no row only contains a single value. 
    Therefore, $C(A_\alpha) \not\cong C(A_\beta)$ for all $\alpha \neq \beta$ and hence $\Gamma_{H_A} \cong \prod_{{\alpha}=1}^z (G_{A_\alpha} \wr \cS_{h_{\alpha}}) \cong \prod_{{\alpha}=1}^z ((\R \times G_{\alpha}) \wr \cS_{h_{\alpha}})$ by Theorem~\ref{thm:embeddablesingle}, Corollary~\ref{cor:embeddableisomorphism}, and the construction of $A$.
\end{proof}

A straightforward adaptation of \cite[Theorem 4.6]{HKDuality} to non-square matrices gives that, for $A \in M_{r\times c}(\F\T)$, the $\cH$-class of $A$ in $M(\T)$ and the $\cH$-class of $A$ in the semigroupoid $M(\F\T)$ are equal.
Thus, we obtain a classification of the Sch\"utzenberger groups of $M(\F\T)$.

\begin{corollary}
The Sch\"utzenberger groups of $M_{n \times m}(\F\T)$, are, up to isomorphism, exactly groups of the form $G \times \R$ where $G \leq \cS_n$ is a finite group with a paired 2-closed permutation representation on $n$ and $m$ points.
\end{corollary}
\begin{proof}
    For any matrix $A \in M_{n \times m}(\F\T)$, $\cB_A$ contains exactly one connected component, so by Corollary~\ref{cor:embeddableisomorphism}, $\Gamma_{H_A} \cong \R \times G$ for some finite group $G$. By Theorem~\ref{thm:SchuzClassification}, $G$ has a paired 2-closed permutation representation on $n$ and $m$ points.
    
    \par Now, note that $\Gamma_{H_A} \cong \R$ for any $A \in M_1(\F\T)$. So, by Lemmas~\ref{lem:reduced} and \ref{lem:SchuzOneDirection}, every group of the form $G \times \R$ where $G \leq \cS_n$ is a finite group with a paired 2-closed permutation representation on $n$ and $m$ points appears as a Sch\"utzenberger group of $M_{n\times m}(\F\T)$.
\end{proof}

\section{Classification of the maximal subgroups} \label{sec:GroupCase}
In this section, we classify the maximal subgroups of $M(\T)$. It is well known that the maximal subgroups are exactly the $\cH$-classes containing idempotents. Thus, as all idempotent matrices are square, classifying the maximal subgroups of $M(\T)$ is equivalent to classifying the maximal subgroups of $M_n(\T)$ for all $n \in \N$. 

The \emph{topological dimension of a subset} $X \subseteq \T^n$ is the largest $k$ for which there exists a dimension $k$ affine space $K \subseteq \T^n$, such that $X \cap K$ has a non-empty relative interior in $K$. For $A \in M_{n \times m}(\T)$, we say the \emph{tropical rank} of $A$ is the topological dimension of $C(A)$.

\begin{lemma}[{\cite[Proposition~3.8 and Lemma~5.6]{GIMMUltimate}}] \label{lem:IdemRank}
    Let $E \in M_n(\T)$ be an idempotent. Then, the row, column and tropical rank coincide.
\end{lemma}

Given the above lemma, we say an idempotent $E \in M_n(\T)$ has \emph{rank} $k$ if $E$ has row rank $k$, and say $E$ has \emph{full rank} if $E$ has rank $n$.

To classify the maximal subgroups of $M_n(\T)$, we require a number of results from the literature that have only been proven for idempotents in $M_n(\F\T)$. The following lemma will be useful to pass these results from $M_n(\F\T)$ to $M_n(\T)$.

\begin{lemma} \label{lem:infinitefiniteidempotent}
    Let $E \in M_n(\T)$ be a full rank idempotent, $N < -\sum_{i,j}|E_{i,j}|$, and for $m \in \N$, let $F_m \in M_n(\F\T)$ be the idempotent power of the matrix obtained by replacing each $-\infty$ in $E$ with $m\cdot N$. Then, for each $m \in \N$,
    \begin{enumerate}[(i)]
        \item $F_m$ has full rank,
        \item $C(F_m) \subseteq C(F_{m+1})$, and
        \item $C(E) \cap \F\T^n = \bigcup_{m\in \N} C(F_m) \setminus (-\infty)^n$
        \end{enumerate}
    where $(-\infty)^n$ is the length $n$ vector only containing $-\infty$.
\end{lemma}
\begin{proof}
    The matrix $E$ has rank $n$ so, by \cite[Proposition 1.6.12]{BMaxLinear} and \cite[Lemma~5.6]{GIMMUltimate}, $F_m$ exists and is an idempotent of rank $n$.
    Recall that $\cG_E$ is the weighted directed graph formed from $E$ and, as $E$ is an idempotent, $E_{i,j}$ is the maximum weight of a path from $i$ to $j$ in $\cG_E$.
    Since $N < -\sum_{i,j}|E_{i,j}|$, if $E_{i,j} \neq -\infty$, then any path in $\cG_{F_m}$ containing an edge weighted by $mN$ has a lower weight than $E_{i,j}$. Similarly, if $E_{i,j} = -\infty$, then any path from $i$ to $j$ in $\cG_{F_m}$ containing at least two edges weighted by $mN$ has a lower weight than the maximum weighted path from $i$ to $j$ only using one such edge. Therefore, 
    \[ (F_m)_{i,j} = \begin{cases}
        E_{i,j} &\text{if } E_{i,j} \neq -\infty, \\
        \max_t(E_{i,t}) + mN + \max_s(E_{s,j}) &\text{otherwise.}
    \end{cases}\]
    
    \par For $m \in \N$, let $E_j$ and $(F_m)_j$ denote the $j$-th column of $E$ and $F_m$ respectively. We show that $C(F_m) \subseteq C(F_{m+1})$ and $C(F_m) \subseteq C(E)$. For all $1 \leq j \leq n$, define $P_j = \{i \colon E_{i,j} = -\infty\}$ and, for $X \in \{E,F_{m+1}\}$, let
    \[ G_j^X = X_j \oplus \left( \bigoplus_{i \in P_j} (F_m)_{i,j} \otimes X_i \right).\]
    Remark that $G_j^X \in C(X)$. We aim to show $(F_m)_j = G_j^X$ for $X \in \{E,F_{m+1}\}$ and all $1 \leq j \leq n$.
    Let $X \in \{E,F_{m+1}\}$. By \cite[Lemma~5.6]{GIMMUltimate}, $X_{i,i} = 0$ for all $i$, as $X$ is a full rank idempotent. Thus, $(F_m)_j \leq G_j^X$, so suppose $G_j^X \not\leq (F_m)_j$. 
    
    \par Then, there exists $1 \leq p \leq n$ and $q \in P_j$ such that 
    \[ (F_m)_{p,j} < (G^X_j)_p = X_{p,j} \oplus ((F_m)_{q,j} \otimes X_{p,q}) = (F_m)_{q,j} \otimes X_{p,q} \]
    where the second equality holds as $X_{p,j} \leq (F_m)_{p,j} < (G^X_j)_p$.
    Then,
    \[ (F_m)_{p,j} < X_{p,q} \otimes (F_m)_{q,j} \leq (F_m)_{p,q} \otimes (F_m)_{q,j} \leq (F_m)_{p,j}\]
    giving a contradiction, where the final inequality holds as $F_m$ is an idempotent.    
    Thus, $G_j^X = (F_m)_j$ for all $j$, and therefore, $C(F_m) \subseteq C(F_{m+1})$ and $C(F_m) \subseteq C(E)$.
    
    \par Finally, suppose $x \in C(E) \cap \F\T^n$. Then,
    \[ x = \lambda_1 E_1 \oplus \cdots \oplus \lambda_n E_n \]
    for some $\lambda_j \in \T$. Recall, we showed that either $(F_m)_{i,j} = E_{i,j}$ or $(F_m)_{i,j} < (m-1)N$. Thus, as $x \in \F\T^n$, if we take $m$ large enough, then
    \[ x = \lambda_1 (F_m)_1 \oplus \cdots \oplus \lambda_n (F_m)_n. \]
    Therefore, $C(E) \cap \F\T^n = \bigcup_{m\in \N} C(F_m) \setminus (-\infty)^n$.
\end{proof}

Let $X \subseteq \T^n$. We say $X$ has \emph{pure dimension} $k$ if every open, within $X$ with the induced topology, subset has topological dimension $k$. Note that, tropically convex sets do not necessarily have pure dimension.

\begin{lemma} \label{lem:puredim}
    Let $E \in M_n(\T)$ be a full rank idempotent. Then, $C(E) \cap \F\T^n$ has pure dimension $n$. 
\end{lemma}
\begin{proof}
By Lemma~\ref{lem:infinitefiniteidempotent}, for each $m \in \N$, there exists a full rank idempotent $F_m \in M_n(\F\T)$ such that $C(F_m) \subseteq C(F_{m+1})$ and $C(E) \cap \F\T^n = \bigcup_{m\in \N} C(F_m) \setminus (-\infty)^n$. By \cite[Theorems~1.1 and 4.5]{IJKDimPolytopes}, $C(F_m) \setminus (-\infty)^n$ has pure dimension $n$ for each $m \in \N$. Thus, $C(E) \cap \F\T^n$ has pure dimension $n$. 
\end{proof}

\begin{lemma}\label{lem:HClassMapping}
    Let $E \in M_n(\T)$ be a full rank idempotent and $A \in H_E$. Consider $C(E) \cap \F\T^n$ as a subset of $\R^n$ with the usual topology. Let $\varphi_A \colon \F\T^n \to \F\T^n$ be left multiplication by $A$. Then,
    \begin{enumerate}[(i)]
        \item $\varphi_A$ maps interior points of $C(E) \cap \F\T^n$ to interior points,
        \item $\varphi_A$ maps boundary points of $C(E) \cap \F\T^n$ to boundary points, and
        \item $\varphi_A$ maps exterior points of $C(E) \cap \F\T^n$ to boundary points.
    \end{enumerate}
\end{lemma}
\begin{proof}
For each $m \in \N$, define $F_m$ as in Lemma~\ref{lem:infinitefiniteidempotent}. Then, each $F_m \in M_n(\F\T)$ is a full rank idempotent such that $C(F_m) \subseteq C(F_{m+1})$ and $C(E) \cap \F\T^n = \bigcup_{m\in \N} C(F_m) \setminus (-\infty)^n$. 

\par As $C(F_m) \subseteq C(F_{m+1})$ for all $m \in \N$, every $x \in \F\T^n$, is either an interior, exterior, or boundary point of $C(F_m)\setminus (-\infty)^n$ for all sufficiently large $m$, we call such a point an \emph{eventually} interior, exterior, or boundary point, respectively.

\par Note that, if $x$ is an eventually interior (resp. boundary) point, then $x$ is an interior (resp. boundary) point of $C(E) \cap \F\T^n$, but if $x$ is an eventually exterior point, then $x$ is either an exterior or boundary point of $C(E) \cap \F\T^n$.

Now, for all $x \in \F\T^n$, $N < 0$ and $1 \leq i,j \leq n$, there exists $m \in \N$ such that 
\[ (Ex)_i \geq x_i = (F_m)_{i,i} + x_i > \max_t(E_{i,t}) + mN + \max_s(E_{s,j}) + x_j. \]
Thus, by the definition of $F_m$ and the proof of Lemma~\ref{lem:infinitefiniteidempotent}, $F_mx = Ex$ for all sufficiently large $m$. Moreover, by Theorem~\ref{thm:units}, $A = PE$ for some unit $P \in M_n(\T)$, so $PF_mx = Ax$ for all sufficiently large $m$. 

\par By \cite[Lemma~5.2]{IJKTropicalGroups}, $\varphi_{PF_m}$ maps interior points of $C(F_m)$ to interior points and maps exterior and boundary points of $C(F_m)$ to boundary points. Thus, $\varphi_A$ maps eventually interior points to eventually interior points, and hence to an interior point of $C(E) \cap \F\T^n$. Similarly, $\varphi_A$ maps eventually exterior and boundary points to eventually boundary points, and hence to boundary points of $C(E) \cap \F\T^n$.
\end{proof}

\begin{proposition} \label{prop:Idemreduced}
    Let $E \in M_n(\T)$ be an idempotent of rank $m$. Then, for every $k \geq m$ there exists an idempotent $F \in M_k(\T)$ of rank $m$ such that $H_E \cong H_F$.
\end{proposition}
\begin{proof}
    By Lemma~\ref{lem:reduced}, there exists $A \in M_k(\T)$ such that $C(A) \cong C(E)$. Thus, by Proposition~\ref{prop:Greens}\textit{(iv)}, $A \cD E$ and hence, $A$ is regular. Therefore, $A \cR F$ for some idempotent $F \in M_k(\T)$. Note that $C(F) = C(A)$ by Proposition~\ref{prop:Greens}. 
    Thus, $C(E) \cong C(F)$, and $H_E \cong H_F$ by Theorem~\ref{thm:auto}.
\end{proof}

Given a full rank idempotent $E \in M_n(\F\T)$, the authors in \cite{IJKTropicalGroups} define $G_E$ to be the set of unit matrices that commute with $E$. Extending this definition to full rank idempotents in $M_n(\T)$, the following theorem shows that, in this case, this definition coincides with our definition of $G_E$ given immediately before Theorem~\ref{thm:units}. This theorem can be seen to be a generalisation of \cite[Theorem~5.3]{IJKTropicalGroups} to full rank idempotents in $M_n(\T)$.

\begin{theorem} \label{thm:GEisCommute}
    Let $E \in M_n(\T)$ be a full rank idempotent. Then, $G_E$ is exactly the set of unit matrices which commute with $E$.
\end{theorem}
\begin{proof}
Let $G_E'$ be the set of unit matrices which commute with $E$. Suppose $P \in G_E'$, then $PE = EP$ and hence, $P \in G_E$. 

So suppose $P \in G_E$, then, by Lemma~\ref{dual}, there exists a unique $Q \in {}_EG$ such that $PE = EQ$. Let $x$ be an interior point of $C(E) \cap \F\T^n$. Then, by Lemma~\ref{lem:HClassMapping}\textit{(i)}, $EQ$ maps $x$ to an interior point of $C(E) \cap \F\T^n$. Moreover, by Lemma~\ref{lem:HClassMapping}\textit{(ii-iii)}, $E$ maps the boundary and exterior points of $C(E) \cap \F\T^n$ to boundary points. Thus, $Q$ maps $x$ to interior points of $C(E) \cap \F\T^n$. In particular, as $E$ fixes every point in $C(E)$, we get that $EQx = Qx$ for all interior points in $C(E) \cap \F\T^n$.

\par By Lemma~\ref{lem:puredim}, $C(E) \cap \F\T^n$ has pure dimension $n$ so every boundary point of $C(E)$ is a limit of interior points. Moreover, as $E$ is a full rank idempotent, every $x \in C(E) \setminus \F\T^n$ is a limit of points in $C(E) \cap \F\T^n$.

Then, as $EQ$ and $Q$ are continuous maps, it follows that $EQx = Qx$ for all $x \in C(E)$. Thus, by Theorem~\ref{thm:auto}, $Q$ is $\T$-semimodule automorphism of $C(E)$. Therefore, $EQE = QE$, and hence
\[ PE = PEE =  EQE = QE,\]
and, as $E$ has full rank, $P = Q$.
\end{proof}

\begin{lemma} \label{lem:idemrestrict}
    Let $E \in M_n(\T)$ be a full rank idempotent. Then, $E|_X$ is a full rank idempotent for all $X \in \cC_E$.
\end{lemma}
\begin{proof}
    Let $\cC_E = \{C_1,\dots,C_k\}$. By \cite[Lemma~5.6]{GIMMUltimate}, $E_{i,i} = 0$ for all $1 \leq i \leq n$, as $E$ is full rank idempotent. Hence, $\omega_i \in C_k$ if and only if $\theta_i \in C_k$. Thus, letting $F = E|_{C_m}$ and $\omega_i,\omega_j \in C_m$, we get that
    \[ F^2_{i,j} = \bigoplus_{\omega_t \in C_m} F_{i,t}F_{t,j} = \bigoplus_{1 \leq t \leq n} E_{i,t}E_{t,j} = (E^2)_{i,j} = E_{i,j} = F_{i,j}\]
    as $E_{i,t} = F_{i,t}$ for any $\omega_t \in C_m$ and $E_{i,t},E_{t,j} = -\infty$ for any $\omega_t \notin C_m$. Thus, $E|_X$ is an idempotent for all $X \in \cC_E$ and, by Remark~\ref{rem:reducedrank}, $E|_X$ has full rank for all $X \in \cC_E$.
\end{proof}

\par We say that a permutation group $G \leq \cS_n$ is called 2-closed if $G$ contains every element of $\cS_n$ which preserves the orbits of ordered pairs. Equivalently, $G$ is the automorphism group of a coloured complete directed graph of $n$ points without loops \cite{CGJKKMNTransPerm}.

\par We are now able to completely classify the groups which, up to isomorphism, arise as a maximal subgroup of $M_n(\T)$.

\begin{theorem} \label{thm:MaximalSubgroupsClassification}
The maximal subgroups of $M_n(\T)$ are, up to isomorphism, exactly groups of the form $\prod_{\alpha=1}^z((\R \times G_\alpha) \wr \cS_{h_\alpha})$ where, for each $1 \leq \alpha \leq z$, there exists $n_\alpha \in \N$ such that
\begin{enumerate}[(i)]
    \item $\sum_{\alpha=1}^z n_{\alpha}h_{\alpha} \leq n$,
    \item each $G_{\alpha}$ is 2-closed on $n_{\alpha}$ points,
    \item at most one $\alpha$ has $n_\alpha = 1$, and
    \item if $G_{\alpha_1},G_{\alpha_2},G_{\alpha_3}$ are trivial for distinct $\alpha_i$, then $n_{\alpha_i} > 2$ for some $i$.
\end{enumerate}
\end{theorem}
\begin{proof}
Let $H_F$ be a maximal subgroup of $M_n(\T)$ with idempotent $F$ of rank $m$. By Proposition~\ref{prop:Idemreduced}, there exists $E \in M_m(\T)$ an idempotent of rank $m$ such that $H_E \cong H_F$. 
Thus, by Theorem~\ref{thm:units}, $H_E \cong G_E$.

\par Consider the graph $\cB_E$, and let $\cup_{\alpha =1}^z \cC_\alpha$ be a partition of $\cC_E$ such that $X,Y \in \cC_\alpha$ if and only if $C(E|_X) \cong C(E|_Y)$. For each $\alpha$, define $\cC_\alpha = \{C_{\alpha,1}, C_{\alpha,2},\dots,C_{\alpha,h_\alpha}\}$ and 
$G_\alpha = \{P \in G_{E|_{C_{\alpha,1}}} \colon P \text{ has eigenvalue } 0 \}$. Then, by Corollary~\ref{cor:embeddableisomorphism}, 
\[G_E \cong \prod_{\alpha=1}^z((\R \times G_\alpha) \wr \cS_{h_\alpha}).\] 

\par By Lemma~\ref{lem:idemrestrict}, $E|_{C_{\alpha,1}}$ is a full rank idempotent for all $\alpha$. So, in particular, $E|_{C_{\alpha,1}} \in M_{n_\alpha}(\T)$ for some $n_\alpha$. So, by Corollary~\ref{cor:embeddableisomorphism} and Theorem~\ref{thm:SchuzClassification} applied to $G_\alpha$, we obtain that $G_\alpha$ has paired 2-closed representation on $n_\alpha$ and $n_\alpha$ points.

By Theorem~\ref{thm:GEisCommute}, $G_E = {}_EG$, and hence, by the definition of 2-closed and paired 2-closed, $G_{\alpha}$ has a 2-closed permutation representation on $n_{\alpha}$ points.
Then, $\sum_{\alpha=1}^z n_{\alpha}h_{\alpha} \leq n$, as $E|_{C} \in M_{n_\alpha}(\T)$ for all $C \in \cC_\alpha$, 
Moreover, by Lemma~\ref{lem:trivialautlem}, at most one $\alpha$ has $n_\alpha = 1$, and if $G_{\alpha_1},G_{\alpha_2},G_{\alpha_3}$ are trivial for distinct $\alpha_i$, then $n_{\alpha_i} > 2$ for some $i$.
Thus, we have shown that all maximal subgroups are of the given form.

\par Now, let $z \in \N$, and for each $1 \leq \alpha \leq z$, let $n_\alpha,h_\alpha \in \N$ be such that $\sum_{{\alpha}=1}^z n_{\alpha}h_{\alpha} \leq n$. Then, let each $G_{\alpha}$ be a 2-closed permutation group on $n_{\alpha}$ points such that at most one $\alpha$ has $n_\alpha = 1$, and if $G_{\alpha_1},G_{\alpha_2},G_{\alpha_3}$ are trivial for distinct $\alpha_i$, then $n_{\alpha_i} > 2$ for some $i$.
We may assume each $G_{\alpha}$ is the automorphism group of a coloured complete directed graph without loops, $D_{\alpha}$, with vertices $\Omega_{\alpha} = \{\omega_{\alpha,1},\dots,\omega_{\alpha,n_\alpha}\}$.
    
    \par For each $\alpha$, let $\Delta_{\alpha,1},\dots,\Delta_{\alpha,k_\alpha}$ be the orbits of $\Omega_\alpha$ under the action of $G_\alpha$ on $D_\alpha$. Then, further colour each edge in $D_\alpha$ to indicate the orbits of the nodes the edge started and ended at, that is, for $x \in \Delta_{\alpha,i}$ and $y \in \Delta_{\alpha,j}$, if the edge from $x$ to $y$ is coloured $\gamma$ instead colour the edge by $\gamma_{\alpha,i,j}$. Note that this does not change the automorphism group of $D_\alpha$ as automorphisms preserve the orbits of the nodes.

\par For each colour used in any $D_\alpha$, pick a distinct real number in the interval $[-1.1, -0.9]$ such that the numbers form a basis for a free abelian group under addition. Then, for each $\alpha$, define $A_\alpha \in M_{n_\alpha}(\T)$ such that the diagonal entries are 0 and, for $s \neq t$, $(A_{\alpha})_{s,t}$ is the number corresponding to the colour of the edge between $\omega_{{\alpha},s}$ and $\omega_{{\alpha},t}$.
However, if there exists $i$ such that $G_{i}$ is trivial with $n_{i} = 2$, let $A_{i} = \left(\begin{smallmatrix}
    0 & 0 \\
    -\infty & 0 
\end{smallmatrix}\right) \in M_2(\T)$, and if there exists $j \neq i$ (or no such $i$ existed) such that $G_{j}$ is trivial with $n_{j} \leq 2$, let $A_{j} = (0) \in M_1(\T)$. Note that $A_i$ and $A_j$ (if they exist) are full rank idempotents. In fact, every other $A_\alpha$ is an idempotent as for any $1 \leq s,t \leq n_\alpha$,
\[ (A_\alpha^2)_{s,t} = \max_k((A_\alpha)_{s,k} + (A_\alpha)_{k,t})\]
where $(A_\alpha)_{s,s} + (A_\alpha)_{s,t} = (A_\alpha)_{s,t} + (A_\alpha)_{t,t} = (A_\alpha)_{s,t} + 0 =(A_\alpha)_{s,t}$ and $(A_\alpha)_{s,k} + (A_\alpha)_{k,s} \leq -0.9 -0.9 < -1.1 \leq (A_\alpha)_{s,t}$ for $k \neq s,t$. Moreover, each $A_\alpha$ can be seen to be full rank as all the diagonal entries are 0 and all the off-diagonal entries are from $[-1.1,-0.9]$.

    Let $E \in M_m(\T)$ be a block diagonal matrix containing $h_\alpha$ copies of $A_\alpha$, for each $\alpha$. It is easy to see that  $m \leq n$ and $E$ is an idempotent of full rank.

    \par We claim that $H_E$ is isomorphic to $\prod_{{\alpha}=1}^z (G_{A_\alpha} \wr \cS_{h_{\alpha}})$. 
    Note that, if $G_{\alpha}$ is trivial and $n_\alpha \leq 2$, then $C(A_{\alpha}) \not\cong C(A_{\beta})$ for any $\beta \neq \alpha$. So suppose otherwise for both $\alpha$ and $\beta$ with $\alpha \neq \beta$, then $A_{\alpha}$ and $A_{\beta}$ have at least two rows and columns. 
    For a contradiction, suppose $C(A_\alpha) \cong C(A_\beta)$ then, by \cite[Theorem 102]{GLecturesMaxPlus}, there exist units $P,Q$ and permutations $\sigma$ and $\tau$ such that $P_{i,\sigma(i)}, Q_{i,\tau(i)} \neq -\infty$ for all $i$ and $PA_\alpha = A_\beta Q$.
    Then, letting $(A_\alpha)_{i,j} = \alpha_{i,j}$ and $(A_\beta)_{i,j} = \beta_{i,j}$ note that
    \[ P_{i,\sigma(i)} + \alpha_{\sigma(i),\tau(j)} = (PA_\alpha)_{i,\tau(j)} = (A_\beta Q)_{i,\tau(j)} = \beta_{i,j} + Q_{j,\tau(j)}.\]
    for all $i,j$. Thus, from the above equation, we obtain that,
    \begin{align*}
        \alpha_{\sigma(1),\tau(2)} + \alpha_{\sigma(2),\tau(1)} - \alpha_{\sigma(1),\tau(1)} - \alpha_{\sigma(2),\tau(2)} &= \beta_{1,2} + \beta_{2,1} -\beta_{1,1} - \beta_{2,2} \\
        &= \beta_{1,2} + \beta_{2,1}.
    \end{align*} 
    However, this gives a contradiction, as the non-zero entries of $(A_\alpha)$ and $(A_\beta)$ form a basis for a free abelian group, they share no non-zero entries, and $\beta_{1,2},\beta_{2,1} \neq 0$. Therefore, $C(A_\alpha) \not\cong C(A_\beta)$ for all $\alpha \neq \beta$ and hence, by Theorem~\ref{thm:embeddablesingle}, Corollary~\ref{cor:embeddableisomorphism}, and the construction of $E$, $H_E \cong \prod_{{\alpha}=1}^z (G_{A_\alpha} \wr \cS_{h_{\alpha}})$.
    
    \par We now show that $G_{A_\alpha} \cong G_\alpha \times \R$, for each $\alpha$. To increase readability, for the remainder of the proof we suppress the dependence on $\alpha$ in all our notation. In particular, we will abuse notation by writing $n$ for $n_\alpha$.

    \par Now suppose $n \leq 2$. If $G$ is the trivial group, then by the definition of $A$ and \cite[Theorem 4.4]{JK2by2}, $G_A \cong \R$. Similarly, if $G \cong \cS_2$, then, $A = \left(\begin{smallmatrix}
        0 & a \\
        a & 0
    \end{smallmatrix}\right)$ for some $a \in [-1.1,-0.9]$, and $G_A \cong G \times \R$ by \cite[Theorem 4.4]{JK2by2}. So we may suppose $n > 2$.

    We now show that $G_A$ consists of exactly the scalings of the permutation matrices corresponding to the permutations of $\Omega$ under the action of $G$ on $D$.
    
    \par For each $\sigma \in \cS_n$, define $\ol{\sigma} \in \cS_{\Omega}$ to be the permutation that maps $\omega_i$ to $\omega_{\sigma(i)}$. Now, let $P$ be the permutation matrix such that $P_{i,\sigma(i)} = 0$ for all $i$.
    By the definition of $A$ and $D$, if $\ol{\sigma} \in G$, then $PAP^{-1} = A$, or equivalently, $PA = AP$. Thus, if $\ol{\sigma} \in G$, then $(\lambda \otimes P)A \in H_A$, and hence $\lambda \otimes P \in G_A$ for all $\lambda \in \F\T$.
    
    Suppose $P \in G_A$ with $P_{i,\sigma(i)} = \lambda_i$ for $1 \leq i \leq n$ and some $\sigma \in \cS_n$. Then, as $P$ commutes with $A$ by Theorem~\ref{thm:GEisCommute}, we get that for all $i,j$,
    \begin{equation} \label{eq:GroupEquation}
        \lambda_i +  A_{\sigma(i),\sigma(j)} = (PA)_{i,\sigma(j)} = (AP)_{i,\sigma(j)} = A_{i,j} + \lambda_j .
    \end{equation}
    
    \par If $\ol{\sigma} \in G$, then, in $D$, the edges $(\omega_i,\omega_j)$ and $(\omega_{\sigma(i)}, \omega_{\sigma(j)})$ have the same colour, and hence $A_{i,j} = A_{\sigma(i),\sigma(j)}$ for all $i$ and $j$. 
    Therefore, equation (\ref{eq:GroupEquation}) gives that $\lambda_i = \lambda_j$ for all $i$ and $j$, so $P$ is a scaling of the permutation matrix corresponding to $\sigma$.
    If $\ol{\sigma} \notin G$, then there exists $\omega_r, \omega_s \in \Omega$ such that the edges $(\omega_r,\omega_s)$ and $(\omega_{\sigma(r)},\omega_{\sigma(s)})$ are different colours in $D$. 
    So, by the definition of $A$, $A_{r,s} \neq A_{\sigma(r),\sigma(s)},$ so let $A_{r,s} = a$ and $A_{\sigma(r),\sigma(s)} = b$.

    \par Remark that, by definition of $A$, if $\{\omega_i,\omega_j\} \subseteq \Delta_l$ for some $l$, then the $i$-th and $j$-th row (resp. column) contain the same entries, but if $\{\omega_i,\omega_j\} \not\subseteq \Delta_l$ for any $l$, then the rows (resp. columns) have no non-zero entries in common.

    \par Suppose $\{\omega_s,\omega_{\sigma(s)}\} \subseteq \Delta_l$ for some $1 \leq l \leq k$. Then, the columns $s$ and $\sigma(s)$ contain the same entries.
    Thus, there exists $1 \leq p \leq n$ such that $A_{p,s} = b$ and $A_{\sigma(p),\sigma(s)} \neq b$, as $A_{r,s} \neq b$ and $A_{\sigma(r),\sigma(s)} = b$. 
    By equation (\ref{eq:GroupEquation}), $\lambda_j = A_{\sigma(r),\sigma(j)} - A_{r,j} + \lambda_r$ for all $j$. So, again by equation (\ref{eq:GroupEquation}), when $i = p$,
    \begin{align*}
        \lambda_p +  A_{\sigma(p),\sigma(j)} &= A_{p,j} + (A_{\sigma(r),\sigma(j)} - A_{r,j} + \lambda_r) &&\text{for all } j \text{, and} \\
        \lambda_p + A_{\sigma(p),\sigma(s)} &= b + (b - a + \lambda_r),
    \end{align*}
    where the second equation is obtained by setting $j = s$.
    Putting these equations together, by equating $\lambda_p$, we obtain that
    \[ a + A_{\sigma(p),\sigma(s)} + A_{\sigma(r),\sigma(j)} + A_{p,j} = b + b + A_{r,j} + A_{\sigma(p),\sigma(j)} \text{ for all } j. \]
    Therefore, as $A_{\sigma(p),\sigma(s)}\neq b \neq a$ and the non-zero entries of $A$ form a basis for a free abelian group, we get that $A_{p,j} = b$ for all $j$. Giving a contradiction as $A_{p,p} = 0 \neq b$.
    
   Thus, $\{\omega_s,\omega_{\sigma(s)}\} \not\subseteq \Delta_l$ for any $1 \leq l \leq k$. Suppose the column $s$ contains $A_{p,s} \neq a,0$ for some $1 \leq p \leq n$, and let $c = A_{p,s}$. Then, by equation (\ref{eq:GroupEquation}), when $i \in \{p,r\}$ and $j = \sigma(s)$,
    \begin{align*}
        (c - A_{\sigma(p),\sigma(s)} + \lambda_s) +  A_{\sigma(p),\sigma^2(s)} &= A_{p,\sigma(s)} + \lambda_{\sigma(s)} \text{, and} \\
        (a - b +\lambda_s) + A_{\sigma(r),\sigma^2(s)} &= A_{r, \sigma(s)} + \lambda_{\sigma(s)}
    \end{align*}
    as $\lambda_i = A_{i,s} - A_{\sigma(i),\sigma(s)} + \lambda_s$ for all $i$, by (\ref{eq:GroupEquation}). Putting these equations together, by equating $\lambda_{\sigma(s)}$, we obtain that
    \[ a + A_{\sigma(r),\sigma^2(s)} + A_{\sigma(p),\sigma(s)} + A_{p,\sigma(s)} = b + c + A_{\sigma(p),\sigma^2(s)} + A_{r, \sigma(s)}.\]
    As $\omega_s$ and $\omega_{\sigma(s)}$ are in different $\Delta$, we have that $c \neq A_{i,\sigma(s)}$ for any $i$. Thus, $c = A_{\sigma(r),\sigma^2(s)}$ as the non-zero entries of $A$ form a basis for a free abelian group.
    Therefore, $\lambda_{\sigma(s)} = a - b + \lambda_s + c - A_{r, \sigma(s)}$ where $A_{r, \sigma(s)} \neq a,c$. 
    Then, by equation (\ref{eq:GroupEquation}) with $j = \sigma(s)$, we get that, for all $i$,
    \[ (A_{i,s} - A_{\sigma(i),\sigma(s)} + \lambda_s) + A_{\sigma(i),\sigma^2(s)} = A_{i,\sigma(s)} + (a - b + \lambda_s + c - A_{r, \sigma(s)}).\]
    Rearranging, we obtain that, for all $i$,
    \[ A_{i,s} + A_{\sigma(i),\sigma^2(s)} + b + A_{r,\sigma(s)} = a + c + A_{\sigma(i),\sigma(s)} + A_{i,\sigma(s)}.\]
    Thus, as the non-zero entries of $A$ form a basis for a free abelian group, $A_{i,s} \in \{a,c\}$ for all $i$, as $A_{r,\sigma(s)} \neq a,c$, giving a contradiction as each column contains a 0 entry. Thus, the $s$ column only contains $a$ and the one 0 entry. Similarly, it can be shown that the $\sigma(s)$ column only contains $b$ and one 0 entry. Moreover, as $\{\omega_s,\omega_{\sigma(s)}\} \not\subseteq \Delta_l$ for any $1 \leq l \leq k$, each row cannot contain both $a$ and $b$. Thus, we obtain a contradiction when $n > 2$, as there must exist a row containing both.
    
    \par Therefore, $P \in G_A$ if and only if $P$ is a scaling of a tropical permutation matrix corresponding to an element of $G$, and hence $G_A \cong G \times \R$.
    \par Finally, as the above process works for all $1 \leq \alpha \leq z$, we have that $H_E$ is isomorphic to $\prod_{{\alpha}=1}^z ((\R \times G_{\alpha}) \wr \cS_{h_{\alpha}})$.
\end{proof}

Let $E \in M_n(\F\T)$ then, by \cite[Theorem 4.6]{HKDuality}, the $\cH$-class of $E$ in $M_n(\T)$ and $M_n(\F\T)$ are equal.
Thus, the classification of the maximal subgroups of $M_n(\F\T)$ now follows easily as a consequence of our main result.

\begin{corollary}
The maximal subgroups of $M_n(\F\T)$ are, up to isomorphism, exactly groups of the form $\R \times G$ where $G$ is a finite group with a 2-closed permutation representation on $n$ points.
\end{corollary}
\begin{proof} 
    For any idempotent $E \in M_{n}(\F\T)$, $\cB_E$ contains exactly one connected component. Therefore, by Corollary~\ref{cor:embeddableisomorphism}, $H_E \cong \R \times G$ for some finite group $G$. Thus, by Theorem~\ref{thm:MaximalSubgroupsClassification}, $G$ has a 2-closed permutation representation on $n$ points.

    \par By the proof of Theorem~\ref{thm:MaximalSubgroupsClassification}, every group of the form $G \times \R$ where $G \leq \cS_n$ is a finite group with a 2-closed permutation representation on $n$ points appears as a maximal subgroup of $M_n(\F\T)$.
\end{proof}

As mentioned in the introduction, the previous result appeared in \cite{IJKTropicalGroups}, however the proof given there contains an error on page 193 where the authors introduce a unit matrix $M$ as part of their argument. There is an unfortunate typo in the definition of this matrix, but even when interpreting this as intended, the equation (5.1) is incorrect. The proof cannot be simply corrected by fixing this equation, since this part of the argument attempts to prove an incorrect claim namely, given an idempotent $E \in M_n(\F\T)$ with entries from $[-0.9,-1.1]$ chosen to generate a free abelian group, it is claimed that $H_E \cong \mathrm{Aut}(\cG_E) \times \R$: but this is false as the example below demonstrates. Nevertheless, the statement of \cite[Theorem~5.10]{IJKTropicalGroups} does turn out to be correct, and the proof of Theorem~\ref{thm:MaximalSubgroupsClassification} in the present article makes use of a similar construction to prove this.
\begin{example}
    Consider the following matrices
    \[ E = \begin{pmatrix}
        0 & a \\
        b & 0
    \end{pmatrix}, \text{ and }
    F = \begin{pmatrix}
        0 & a & c & a \\
        b & 0 & b & c \\
        c & a & 0 & a \\
        b & c & b & 0
    \end{pmatrix}\]
    where $a,b,c \in [-0.9,-1.1]$ are chosen to be generators of a free abelian group of rank 3. Then, $\mathrm{Aut}(\cG_E) \cong \{e\}$ and $\mathrm{Aut}(\cG_F) \cong \cS_2 \times \cS_2$.
    However, it can be easily seen that $PE = EP$ and $QF = FQ$ where
    \[ P = \begin{pmatrix}
        -\infty & a \\
        b & -\infty        
    \end{pmatrix} \text{ and } Q = \begin{pmatrix}
        -\infty & a & -\infty & -\infty \\
        -\infty & -\infty & b & -\infty \\
        -\infty & -\infty & -\infty & a \\
        b & -\infty & -\infty & -\infty
    \end{pmatrix}.\]
    In fact, it can be shown that $H_E \cong \cS_2 \times \R$ and $H_F \cong D_4 \times \R$, where $D_4$ is the dihedral group of order 8.
\end{example}
By correcting the proof of the above theorem, we regain the following corollary.
\begin{corollary}
    For any finite group $G$, $G \times \R$ arises as a maximal subgroup of $M_n(\F\T)$ for all $n \geq |G|$.
\end{corollary} 
This can be seen to be sharp by considering cyclic groups of prime degrees as their permutation degree is equal to their order. The authors of \cite{IJKTropicalGroups} state that the above bound is sharp for the alternating group of degree 4, $\mathrm{Alt}(4)$. However, $\mathrm{Alt}(4) \times \R$ arises as a maximal subgroup of $M_{10}(\F\T)$ as $\mathrm{Alt}(4)$ has a 2-closed permutation representation on 10 points, generated by the following permutations,
\[(1, 3, 2)(5, 10, 7)(6, 8, 9), \ (1, 4)(2, 3)(6, 10)(7, 8), \ (1, 3)(2, 4)(5, 9)(6, 10)\]
found using \textsc{Magma} \cite{Magma}.

We have now classified the maximal subgroups and the Sch\"utzenberger groups of $M_n(\T)$, thus it poses the question of whether there exist groups that appear as a Sch\"utzenberger groups of $M_n(\T)$ but do not appear as a maximal subgroup. 

\begin{corollary}
    There exists a finite group $G$ such that $G \times \R$ appears as Sch\"utzenberger group of $M_{16}(\F\T)$, but does not appear as a maximal subgroup of $M_{16}(\F\T)$.
\end{corollary}
\begin{proof}
Let $g_1\dots,g_{12}$ be the 12 distinct elements of $\mathrm{Aut}(4)$ and let $a,b,c,d \in \F\T$ generate a free abelian group of rank 4. Then, let $\mathrm{Alt}(4)$ act on the vector $V = (a,b,c,d)^T$ in the obvious way.
Define $A \in M_{4 \times 12}(\T)$ such that the $i$th column of $A$ is given by $g_i \cdot V$. Then, by the proof of Theorem~\ref{thm:SchuzClassification}, it can be seen that $\Gamma_A$ is isomorphic to $\mathrm{Alt}(4) \times \R$ as the action of any $g_i$ on $A$ only permutes the columns by definition. Now, consider the matrix
\[ B = \left(\begin{array}{c|c}
    A & 0 \\
    \hline
    0 & A^T
\end{array}\right)\]
where $0$ denotes that the entire block is equal to 0. Then, it can be shown that $\Gamma_B \cong (\mathrm{Alt}(4) \times \mathrm{Alt}(4)) \times \R$. However, $\mathrm{Alt}(4) \times \mathrm{Alt}(4)$ does not have a 2-closed permutation representation of 16 points. Therefore, by Theorem~\ref{thm:MaximalSubgroupsClassification}, $(\mathrm{Alt}(4) \times \mathrm{Alt}(4)) \times \R$ does not appear as a maximal subgroup of $M_{16}(\T)$.
\end{proof}

\begin{question}
    What is the minimum $n \in \N$ such that the above corollary holds for $M_n(\F\T)$ and $M_n(\T)$?
\end{question}

\bibliographystyle{plain}
\bibliography{Bib.bib}

\begin{thebibliography}{10}

\bibitem{ARStylic}
T.~Aird and D.~Ribeiro.
\newblock Tropical representations and identities of the stylic monoid.
\newblock {\em Semigroup Forum}, 106(1):1--23, 2023.

\bibitem{Magma}
W.~Bosma, J.~Cannon, and C.~Playoust.
\newblock The {M}agma algebra system. {I}. {T}he user language.
\newblock {\em J. Symbolic Comput.}, 24(3-4):235--265, 1997.
\newblock Computational algebra and number theory (London, 1993).

\bibitem{BMaxLinear}
P.~Butkovi{\v{c}}.
\newblock {\em Max-linear systems: theory and algorithms}.
\newblock Springer Science \& Business Media, 2010.

\bibitem{BSSGenExtemeBases}
P.~Butkovič, H.~Schneider, and S.~Sergeev.
\newblock Generators, extremals and bases of max cones.
\newblock {\em Linear Algebra and its Applications}, 421(2):394--406, 2007.
\newblock Special Issue in honor of Miroslav Fiedler.

\bibitem{CJKMPlacticLikeRep}
A.~J. Cain, M.~Johnson, M.~Kambites, and A.~Malheiro.
\newblock Representations and identities of plactic-like monoids.
\newblock {\em J. Algebra}, 606:819--850, 2022.

\bibitem{CGJKKMNTransPerm}
P.~J. Cameron, M.~Giudici, G.~A. Jones, W.~M. Kantor, M.~H. Klin, D.~Marušič, and L.~A. Nowitz.
\newblock Transitive permutation groups without semiregular subgroups.
\newblock {\em Journal of the London Mathematical Society}, 66(2):325–333, 2002.

\bibitem{CPSemigroups}
A.~H. Clifford and G.~B. Preston.
\newblock The algebraic theory of semigroups, vol. 1.
\newblock {\em Amer. Math. Soc. Surveys}, 7(1961):1967, 1961.

\bibitem{GLecturesMaxPlus}
S.~Gaubert.
\newblock Two lectures on max-plus algebra.
\newblock {\em Proceedings of the 26th Spring School of Theoretical Computer Science}, pages 83--147, 1998.

\bibitem{GJNMatrixOverSemirings}
V.~Gould, M.~Johnson, and M.~Naz.
\newblock Matrix semigroups over semirings.
\newblock {\em International Journal of Algebra and Computation}, 30(02):267--337, 2020.

\bibitem{GIMMUltimate}
P.~Guillon, Z.~Izhakian, J.~Mairesse, and G.~Merlet.
\newblock The ultimate rank of tropical matrices.
\newblock {\em Journal of Algebra}, 437:222--248, 2015.

\bibitem{HKDuality}
C.~Hollings and M.~Kambites.
\newblock Tropical matrix duality and green’s {$\mathscr{D}$} relation.
\newblock {\em Journal of the London Mathematical Society}, 86(2):520–538, 2012.

\bibitem{IJKDimPolytopes}
Z.~Izhakian, M.~Johnson, and M.~Kambites.
\newblock Pure dimension and projectivity of tropical polytopes.
\newblock {\em Advances in Mathematics}, 303:1236--1263, 2016.

\bibitem{IJKTropicalGroups}
Z.~Izhakian, M.~Johnson, and M.~Kambites.
\newblock Tropical matrix groups.
\newblock {\em Semigroup Forum}, 96(1):178–196, 2017.

\bibitem{JK2by2}
M.~Johnson and M.~Kambites.
\newblock Multiplicative structure of 2×2 tropical matrices.
\newblock {\em Linear Algebra and its Applications}, 435(7):1612--1625, 2011.
\newblock Special Issue dedicated to 1st Montreal Workshop.

\bibitem{JKGreensJ}
M.~Johnson and M.~Kambites.
\newblock Green’s {$\mathscr{J}$}-order and the rank of tropical matrices.
\newblock {\em Journal of Pure and Applied Algebra}, 217(2):280--292, 2013.

\bibitem{JKPlacticRep}
M.~Johnson and M.~Kambites.
\newblock Tropical representations and identities of plactic monoids.
\newblock {\em Trans. Amer. Math. Soc.}, 374(6):4423--4447, 2021.

\bibitem{JCombinatorics}
M.~Joswig.
\newblock {\em Essentials of tropical combinatorics}, volume 219.
\newblock American Mathematical Society, 2021.

\bibitem{MSGeometry}
D.~Maclagan and B.~Sturmfels.
\newblock {\em Introduction to tropical geometry}, volume 161.
\newblock American Mathematical Society, 2021.

\bibitem{STropical}
Y.~Shitov.
\newblock Tropical matrices and group representations.
\newblock {\em Journal of Algebra}, 370:1--4, 2012.

\end{thebibliography}

\end{document}